\pdfoutput=1
\RequirePackage{ifpdf}
\ifpdf 
\documentclass[pdftex]{sigma}
\else
\documentclass{sigma}
\fi

\numberwithin{equation}{section}

\newtheorem{Theorem}{Theorem}[section]
\newtheorem*{Theorem*}{Theorem}

\newtheorem{Lemma}[Theorem]{Lemma}

 { \theoremstyle{definition}

\newtheorem{Example}[Theorem]{Example}
\newtheorem{Remark}[Theorem]{Remark} }

\begin{document}

\allowdisplaybreaks

\newcommand{\arXivNumber}{2404.15767}

\renewcommand{\PaperNumber}{075}

\FirstPageHeading

\ShortArticleName{Isomonodromy and Painlev\'e Type Equations, Case Studies}

\ArticleName{Isomonodromy and Painlev\'e Type Equations,\\ Case Studies}

\Author{Marius VAN DER PUT and Jaap TOP}

\AuthorNameForHeading{M.~van~der~Put and J.~Top}

\Address{Bernoulli Institute, Nijenborgh 9, 9747 AG~Groningen, The Netherlands}
\Email{\href{mailto:m.van.der.put@rug.nl}{m.van.der.put@rug.nl}, \href{mailto:j.top@rug.nl}{j.top@rug.nl}}
\URLaddress{\url{http://www.math.rug.nl/~top/}}

\ArticleDates{Received July 02, 2024, in final form August 28, 2025; Published online September 11, 2025}

\Abstract{There is an abundance of equations of Painlev\'e type besides the classical Painlev\'e equations. Classifications have been computed by the Japanese school. Here we consider Painlev\'e type equations induced by isomonodromic families of linear ODE's having at most~${z=0}$ and $z=\infty$ as singularities. Requiring that the formal data at the singularities produce isomonodromic families parametrized by a single variable $t$ leads to a small list of hierarchies of cases. The study of these cases involves Stokes matrices
and moduli for linear ODE's on the projective line. Case studies reveal interesting families of linear ODE's and Painlev\'e type equations. However, rather often the complexity (especially of the Lax pair) is too high for either the computations or for the output. Apart from classical Painlev\'e equations one rediscovers work of Harnad, Noumi and Yamada. A hierarchy, probably new, related to the classical $P_3(D_8)$, is discovered. Finally, an amusing ``companion'' of $P_1$ is presented.}


\Classification{33E17; 14D20; 14D22; 34M55}

 \section{Introduction and background}\label{section0}

\subsection{Introduction}\label{FlNe}
Every classical Painlev\'e equation can be obtained by an isomonodromy of a family $\mathcal{M}$ of linear differential equations over the differential field $\mathbb{C}(z)$ \big(or connections on $\mathbb{P}^1$\big). The monodromy space $\mathcal{R}$, associated to $\mathcal{M}$, consists of all possibilities for the ordinary monodromy and the Stokes data of the solutions for the linear differential equations
belonging to $\mathcal{M}$.

An isomonodromic subfamily of $\mathcal{M}$ is defined by ``the monodromy is constant'' condition, in other words, it is
a fibre of the Riemann--Hilbert map ${\rm RH}\colon \mathcal{M}\rightarrow \mathcal{R}$ which sends a connection in $\mathcal{M}$
to its monodromy data in $\mathcal{R}$.
An isomonodromic subfamily turns out to be given by first-order, nonlinear, differential equations
(called here the Painlev\'e vector field) which can be transformed into a single higher-order nonlinear differential equation.

For example (see also \cite[Section~4.9]{vdP-Sa}),
\[
\frac{{\rm d}^2q}{{\rm d}t}=\frac{\bigl(\frac{{\rm d}q}{{\rm d}t}\bigr)^2}{2q}+4q^2+2tq-\frac{\theta_0^2}{2q},
\]
 which is the Flaschka--Newell
form $P_{2,fn}$ of the second Painlev\'e equation, is obtained from the family \[z\frac{{\rm d}}{{\rm d}z}+ \left(\begin{matrix} p & \frac{p^2-\theta_0^2-2qtz+q^2z+qz^2}{q}\\ z-q & -p \end{matrix}\right). \]
This 4-dimensional family of operators is defined by the conditions:
the trace is zero, the singularities are $z=0$, which is regular singular, and $z=\infty$ which is irregular singular and
has (generalized) eigenvalues $\pm \bigl(z^{3/2}+tz^{1/2}\bigr)$.

 For this example, the monodromy space $\mathcal{R}$ consists of the ordinary monodromy $\operatorname{mon}_0$ at $z=0$ and Stokes
data at $z=\infty$. It can be shown that it consists of Stokes matrices
$\bigl(\begin{smallmatrix} 1 & 0\\a_1 &1\end{smallmatrix}\bigr), \bigl(\begin{smallmatrix} 1 & a_2\\0 &1\end{smallmatrix}\bigr), \bigl(\begin{smallmatrix} 1 & 0\\a_3 &1\end{smallmatrix}\bigr)$ for directions $\{d_i\}$ with $1>d_1>d_2>d_3>0$.
There is one relation, the monodromy identity:
$\operatorname{mon}_0$ is conjugated to
\smash{$\bigl(\begin{smallmatrix} 0 & -1\\1 &0\end{smallmatrix}\bigr)\bigl(\begin{smallmatrix} 1 & 0\\a_1 &1\end{smallmatrix}\bigr)\bigl(\begin{smallmatrix} 1 & a_2\\0 &1\end{smallmatrix}\bigr)\bigl(\begin{smallmatrix} 1 &0\\a_3 &1\end{smallmatrix}\bigr)$},
where $\bigl(\begin{smallmatrix} 0 & -1\\1 &0\end{smallmatrix}\bigr)$ is the formal monodromy. Hence~${\mathcal{R}=\mathbb{C}^3}$.
It can be shown that the fibres of the Riemann--Hilbert map ${\rm RH}\colon\mathcal{M}\rightarrow \mathcal{R}$ are parametrized
by $t$. A fibre has the form
\[
z\frac{{\rm d}}{{\rm d}z}+ \left(\begin{matrix} p & \frac{p^2-\theta_0^2-2qtz+q^2z+qz^2}{q}\\ z-q & -p \end{matrix}\right)
\]
with now $p$, $q$ functions of $t$ and $\theta_0$ a parameter (independent of $t$).

The functions $p$, $q$ of $t$ cannot directly be computed by the Riemann--Hilbert map, since the latter
 is (in general) highly transcendental and not computable. However, the condition ``isomonodromic'' is equivalent to
the above operator commuting with an operator of the form~${\frac{{\rm d}}{{\rm d}t}+B(z,t)}$, where the $2\times 2$ matrix
$B(z,t)$ depends analytically on $t$ and rationally on~$z$. This is the Lax pair condition and leads to the nonlinear
differential equations
\[
\frac{{\rm d}q}{{\rm d}t}=2p , \qquad \frac{{\rm d}p}{{\rm d}t}=2q^2+tq+\frac{p^2-\theta_0^2/4}{q}
\]
 (the Painlev\'e vector field)
and as consequence to $P_{2,fn}$. Finally, the system of equations for~$q$,~$p$ is Hamiltonian with function
\[
H=\frac{-p^2+\theta_0^2/4}{q}+q^2+tq.
\]

The parameter space $\mathcal{P}$ for $\mathcal{R}$ consists of the data for the local formal monodromy at $z=0$ and $z=\infty$. The parameter space $\mathcal{P}^+$ for $\mathcal{M}$ consists of the data of the residue matrices of the connection at $z=0$ and $z=\infty$. There is an obvious exponential map $\mathcal{P}^+\rightarrow \mathcal{P}$
commuting with the Riemann--Hilbert map ${\rm RH}\colon \mathcal{M}\rightarrow \mathcal{R}$.

In the above example, $\mathcal{P}^+$ is given by $\theta_0$ and $\mathcal{P}$ is given by ${\rm e}^{2\pi {\rm i} \theta_0}$.

 Apart from the classical cases $P_1-P_6$, there is an abundance of families $\mathcal{M}$ of linear differential equations over $\mathbb{C}(z)$ producing Painlev\'e type equations. The Japanese school has an extensive literature on equations of Painlev\'e type and also developed classifications.

 We apologize for citing only a few items from the extensive
literature, \cite{F-I-K-N,H-K-N-S,J-M-U,K3,K2,K4,K-N-S,Mi,N-Y,O-O,O1}. In the literature, other methods than isomonodromy,
 e.g., middle convolution and constructions with Hamiltonians, are used for producing Painlev\'e type equations.

Here, we {\it modestly restrict ourselves} to classifying and studying rather special cases of~$\mathcal{M}$, namely assuming that at most $z=0$ and $z=\infty$ are singular and assuming that the fibres of the Riemann--Hilbert map
${\rm RH}\colon\mathcal{M}\rightarrow \mathcal{R}$ are parametrized by one variable~$t$, called the {\it time variable}.
In general there are more times variables.

 One reason for this restriction is that the quantum differential equations (see \cite{C-D-G,G-G-I,G2,G1}), associated to algebraic varieties have two singularities $z=0$ (regular singular) and $z=\infty$ (irregular singular). A further reason is that the theory of Stokes matrices,
and the algorithm we describe to make $\mathcal{M}$ explicit in the case of at most two singularities, provide most of the information for $\mathcal{R}$ and $\mathcal{M}$ used here.
Furthermore, we are also interested in {\it hierarchies} of families instead of individual families.

{\it The aim of this paper} is to find these families $\mathcal{M}$ and classify them by their formal singularities at $z=0$
and $z=\infty$ (see the list in Section~\ref{list}).

For each item in this list, we try to make the family of connections $\mathcal{M}$ explicit by computing a~matrix differential operator
$z\frac{{\rm d}}{{\rm d}z}+A$. This can be rather involved. Further we want to make the monodromy space $\mathcal{R}$
explicit, to produce a Lax pair and compute the Painlev\'e type equation (or vector field), produce a Hamiltonian and find
an identification (if this exists) with some known Painlev\'e type equation.

Due to complexity, we obtained for many items in the list of Section~\ref{list}, only partial information. However,
we highlight here:
 {\it A list of the most detailed and interesting explicit cases}.
 The formulas give representatives for the Galois orbit(s) of the eigenvalues and $n$ is the rank of the connection.
 \begin{itemize}\itemsep=0pt
\item Section~\ref{two}: $z^{2/n}+tz^{1/n}$, $n\geq 3$. We rediscover the hierarchy studied by~Noumi and~Yamada~\cite{N-Y}.
The spaces $\mathcal{M}$ and $\mathcal{R}$ and the fibres of $\mathcal{M}\rightarrow \mathcal{R}$ are made explicit.
For any $n\geq 3$, explicit formulas for the Lax pair and the Painlev\'e vector field are computed.
A~Hamiltonian is computed for $n=3,4,6$. The Painlev\'e equations for $n=3$ and $n=4$ are identified with~$P_4$ and $P_5$. An identification of the Hamiltonian for $n=6$ is missing.

\item Section~\ref{5.1}: $z^{1/2}$, $tz^{1/2}$, $n=4$.
Apart from the Hamiltonian all data are made explicit. The fibres of the map $\mathcal{R}\rightarrow \mathcal{P}$ are
affine cubic surfaces. From their equation we expected, in view of the list \cite[pp.~26--27]{vdP-Sa}, that this case is a
 pull back of $P_4$. Instead,
 the computation in \cite[Section~1]{Dz} provides an explicit identification with the sixth Painlev\'e equation in Okamoto's form.

\item Section~\ref{seven}: $z$, $tz$, $(-1-t)z$, $n=3$. Except for the Hamiltonian, all data (including the Stokes matrices) are made explicit. However the expected identification with $P_6$ is not verified.

In fact, we rediscover the family found by
Harnad and studied by Mazzocco, Degano and Guzetti \cite{D-G,H,Maz}. They provide the verification
of the equivalence to~$P_6$.

\item Section~\ref{nine}: $z^2$, $-z^2-tz$, $tz$, $n=3$. The results, including an explicit Hamiltonian, are complete. The
Painlev\'e type equation is a second-order explicit equation and therefore related to one of the classical Painlev\'e
equations. From the cubic equation of the fibres of $\mathcal{R}\rightarrow \mathcal{P}$, we expected a relation
with $P_1$. Instead, the computation in \cite[Section~2]{Dz} provides an explicit identification with the fourth Painlev\'e equation in Okamoto's form.

\item Section~\ref{twelve}: $z^{-1/n}$ and $tz^{1/n}$, $n\geq 2$. This hierarchy is probably new. For $n=2$, this defines the standard family leading to $P_3(D_8)$. For general $n$, the spaces of connections and of monodromy
 are made explicit. The explicit formulas for the Lax pairs and the Painlev\'e vector field have a structure,
similar to the ones for the Noumi--Yamada family in Section~\ref{two}.

 For $n=3$, we claim an identification with an item in the classification
 by Kawakami~\cite[p.~35]{K3}.

\item Section~\ref{thirteen}: $z^{5/2}+tz^{1/2}$, $n=2$.
What we like to call ``a companion of $P_1$'' is given by the family of connections of rank 2 with a regular singularity
at $z=0$ and an irregular singularity at~${z=\infty}$ with generalized eigenvalues $\pm \bigl(z^{5/2}+tz^{1/2}\bigr)$.

We recall that the standard family of connections for $P_1$ has the same definition except for the assumption
``$z=0$ is regular''.

 For this companion of $P_1$, all data are computed, including the Painlev\'e equation which is a nonlinear, explicit,
 fourth-order differential equation.

 The above family is not present in the list of Section~\ref{list}, since it is a subfamily of the {\it family
 with two time variables } $t_1$, $t_2$. This family is given by $z=0$ is regular singular and $z=\infty$ has generalized eigenvalues $\pm \bigl( z^{5/2}+\frac{t_1}{2}z^{3/2}+\frac{t_2}{2}z^{1/2} \bigr)$.

 Also in this case an operator $z\frac{{\rm d}}{{\rm d}z}+A$, representing $\mathcal{M}$, is computed. Further, $\mathcal{R}$ and the fibres $\mathcal{R}\rightarrow \mathcal{P}$ are made explicit. The Lax pair equations have now the form $\bigl[\smash{z\frac{{\rm d}}{{\rm d}z}+A, \frac{{\rm d}}{{\rm d}t_i}}+B_i\bigr]=0$ for $i=1,2$ for certain matrices $B_1$, $B_2$ depending on $z$, $t_1$, $t_2$. The Painlev\'e vector field, solution of the Lax pairs, is computed.
 It seems that this companion of $P_1$ and its extension to a ``two time variables system'' are new.
 \end{itemize}

\subsection{Background}
For the convenience of the reader, we describe terminology and results concerning the
formal classification of differential modules, irregular singularities and Stokes matrices etc.
More details can be found in~\cite{vdP-Si}.

\subsubsection{The formal classification of differential modules}

Differential modules $M$ over the field $\mathbb{C}\bigl( \!\bigl(z^{-1}\bigr)\!\bigr)$ are classified
in terms of tuples $(V, \{V_q\}_q, \gamma)$. A~tuple consists of a complex vector space $V$ of dimension $n$ with additional structure.
The $q$'s denote elements of $ \bigcup_{r\geq 1} z^{1/r}\mathbb{C}\bigl[\!\bigl[z^{1/r}\bigr]\!\bigr]$. For each $q$ there is given a linear
subspace $V_q\subset V$ and~${V=\oplus V_q}$. The {\it eigenvalues } are the finitely many $q_1,\dots ,q_r$ with $V_{q}\neq 0$.
The multiplicity~${m=m(q)}$ is the dimension of $V_q$. The dimension of $V$ is therefore equal to $\sum m(q_j)$.
One writes $(q)_m$ to denote an eigenvalue $q$ with multiplicity $m$.

The {\it ramification index} $e$ is the smallest positive integer such that $q_j\in z^{1/e}\mathbb{C}\bigl[\!\bigl[z^{1/e}\bigr]\!\bigr]$ for all~$j$. The degree of
$q_j$ is the highest (rational) power of $z$ occurring in $q_j$. The {\it Katz invariant} $\kappa$ is the maximum
of the degrees of the $q_j$'s. The Galois group of $\bigcup_{r\geq 1}\mathbb{C}\bigl(\!\bigl( z^{-1/r}\bigr)\!\bigr)/ \mathbb{C}\bigl(\!\bigl(z^{-1}\bigr)\!\bigr)$ has a~topological generator $\sigma$ which acts by $\sigma \bigl(z^\lambda\bigr)={\rm e}^{2\pi {\rm i} \lambda}z^\lambda$ for $\lambda \in \mathbb{Q}$.

Further, $\gamma$, the {\it formal monodromy}, is an automorphism of $V$ and has the property
$\gamma (V_q)=V_{\sigma(q)}$ for all~$q$.

This classification of differential modules $M$ over $\mathbb{C}\bigl(\!\bigl(z^{-1}\bigr)\!\bigr)$
 by the tuples $(V,\{V_q\}_q,\gamma)$ is based upon the fact that the solutions of $M$ \big(where $M$ is represented in the form of an ordinary
 scalar linear differential equation with coefficients in $\mathbb{C}\bigl(\!\bigl(z^{-1}\bigr)\!\bigr)$\big) \color{black} can be written as a sum of
 expressions~\smash{$\exp\bigl(\int q\frac{{\rm d}z}{z}\bigr)\cdot G$} with $q$ as above and $G$ a combination of formal power series
 in roots of $z$ and $\log (z)$. The space $V$ associated to $M$ is the space of these formal or symbolic expressions. It has a natural decomposition
 as $\oplus V_q$. Further, the formal monodromy $\gamma$, is given by $\sigma$ applied to the above expressions.
The functor $M\mapsto (V, \{V_q\}_q, \gamma)$ is an equivalence of Tannakian categories.

{\it Let a family of data for eigenvalues $q_1,\dots, q_r$, multiplicities and formal monodromies $\gamma$ be given}. This gives rise to a family of formal differential operators $z\frac{{\rm d}}{{\rm d}z}+F$.

\begin{Example}[see Section~\ref{six}] Eigenvalues $c_1z$, $c_2z$, $c_3z$ with multiplicities 1 and formal monodromy
\[
\gamma=\begin{pmatrix} b_1 & & \\ & b_2 & \\ & & b_3 \end{pmatrix},
\]
 combine to
the formal differential operator $z\frac{{\rm d}}{{\rm d}z}+F$, where
\[
F=\begin{pmatrix} c_1z+ a_1 & & \\ & c_2z+a_2 & \\ & & c_3z+a_3 \end{pmatrix}
\]
and $b_j={\rm e}^{2 \pi {\rm i} a_j}$ for $j=1,2,3$.
\end{Example}

The space of connections $\mathcal{M}$ is the ``universal family'' (details later), represented by a
family of differential operators $z\frac{{\rm d}}{{\rm d}z}+A$ over $\mathbb{C}(z)$ with singular points $z=0$ and $z=\infty$.
The condition is that this operator is at $z=\infty$ is formally equivalent to $z\frac{{\rm d}}{{\rm d}z}+F$. Moreover, at $z=0$ the operator
should be regular, or regular singular, or formally equivalent to another given formal operator.

\subsubsection[Details on Stokes matrices and construction of the space R]{Details on Stokes matrices and construction of the space $\boldsymbol{\mathcal{R}}$}
Let $M$ be a differential module over $\mathbb{C}\bigl(\bigl\{ z^{-1}\bigr\}\bigr)$, the field of convergent power series at $z=\infty$.
Write $(V, \{V_{q_k}\}, \gamma)$ for its formal classification, i.e., the classification of
\smash{$\widehat{M}=\mathbb{C}\bigl(\!\bigl(z^{-1}\bigr)\!\bigr)\otimes M$}. The formal solutions of \smash{$\widehat{M}$} lift, by multisummation, to solutions of $M$ on sectors
at $z=\infty$. The jumps of these solutions from one sector to another are measured by
Stokes matrices $\{{\rm St}_d\}$ at the singular directions $d$. In fact, $M$ is classified by its formal classification and the
Stokes matrices. Moreover, the formalism, detailed below, produces all possibilities for the Stokes data and
determines therefore the structure of $\mathcal{R}$.

For each difference $q_k-q_l$, $ k\neq l$ of eigenvalues, one considers
a solution $y\neq 0$ of $z\frac{{\rm d}y}{{\rm d}z}=q_k-q_l$. The singular directions $d\in \mathbb{R}$ for $q_k-q_l$
are defined by the condition that $y\bigl({\rm e}^{2\pi {\rm i} d}r\bigr)$ tends to zero for $r\rightarrow +\infty$ with maximal speed (maximal descent).

The Stokes matrix
${\rm St}_d\in {\rm GL}(V)$ for direction $d$ reads $\mathbf{1}_V +\sum _{k,l}m_{k,l}$, where the sum is taken over the pairs such that $d$ is singular for $q_k-q_l$ and $m_{k,l}$ denotes a linear map \[
V\stackrel{\rm projection}{\rightarrow}V_{q_k}\rightarrow V_{q_l}\subset V.\]

 We note that ${\rm St}_d={\bf 1}_V$ if $d$ is not a singular direction and that ${\rm St}_{d+1}=\gamma^{-1} {\rm St}_d \gamma$ holds for~${ d\in \mathbb{R}}$.

 Therefore, the Stokes data can be identified with the space of all Stokes matrices ${\rm St}_d$
with $d\in [0,1)$ and can be identified with a vector space of dimension
\[
N:=\sum _{k\neq l} \deg (q_k-q_l)\cdot \dim V_{q_k}\cdot \dim V_{q_l}.\]

The formal data combined with the data of the Stokes matrices classify the analytic singularity at $z=\infty$. In particular,
$\operatorname{mon}_\infty$, the topological monodromy at $z=\infty$, is equivalent to the product $\gamma \circ {\rm St}_{d_s}\circ \dots \circ {\rm St}_{d_1}$,
where $d_s>\dots >d_1$ are the singular directions in $[0,1)$. This property will be called {\it the monodromy identity}. For the construction of
$\mathcal{R}$, we have to consider various cases.

 (i) {\it $z=0$ is regular singular}. Given are $V=V_{q_1}\oplus \cdots \oplus V_{q_r}$, an action
of $\sigma$ on $\{q_1,\dots , q_r\}$, the singular directions $1>d_s> \cdots >d_1\geq 0$ with the corresponding differences
$q_k-q_l$. We note, in passing, that the highest coefficient of a difference $q_k-q_l$ may depend on $t$. In such a case the
singular directions also depend on $t$. In the cases that we computed, $\mathcal{R}$ itself is independent of~$t$.\looseness=-1

In general, $\mathcal{R}$ is defined as the set of equivalence classes of all possibilities for the Stokes data and the
topological monodromies. In the present case, $\mathcal{R}$ consists of the equivalence classes of all possible tuples
$(\gamma, {\rm St}_{d_s},\dots , {\rm St}_{d_1})\in {\rm SL}(V)^{s+1}$, where, by assumption, $\gamma$ is supposed to have
distinct eigenvalues. We now make $\mathcal{R}$ explicit.

{\it Let $(V,\{V_q\},\gamma ,\{{\rm St}_d\})$ be given}. The action of $\sigma $ on the eigenvalues has orbits (i.e., Galois orbits)
$\overline{Q}_1,\dots ,\overline{Q}_r$ and $\overline{Q}_i=\{q_{i,0},\dots ,q_{i,\ell_i-1}\}$ for all $i$. Let $d_i=\dim V_{q_{i,0}}$.
Put $d=\sum _{i=1}^rd_i $ and as before we write
\[
N=\sum \deg ( q_{i,j}-q_{k,l})\cdot \dim (V_{q_{i,j}})\cdot \dim (V_{q_{k,l}}).\]

\begin{Lemma} \label{lemma01} The monodromy space $\mathcal{R}$ is isomorphic to the quotient of the space $(\mathbb{C}^*)^{d-1}\times \mathbb{C}^N$ by the
 action of a~group isomorphic to $(\mathbb{C}^*)^{d-1}$. This quotient has an open, affine, dense subspace
 isomorphic to $(\mathbb{C}^*)^{d-1}\times \mathbb{C}^{N-d+1}$. In particular, $\dim \mathcal{R}=N$.
\end{Lemma}
\begin{proof} There is no restriction on the possibilities for the~\smash{${\rm St}_{d_j}$}. This produces the vector space~$\mathbb{C}^N$.
In order to make the restrictions on $\gamma$ explicit, we consider a Galois orbit, here written as,
$\overline{Q}=\{q_0,\dots ,q_{\ell -1}\}$ with $\dim V_{q_0}=f$. Choose a basis
$e_{1},\dots, e_{f}$ of eigenvectors for the action of $\gamma^\ell$ on $V_{q_0}$. Let $\alpha_1,\dots ,\alpha_f$ denote the
distinct eigenvalues. This basis is unique up to permuting and scaling of the basis vectors. Consider, for $i=1,\dots, \ell -1$, the basis of $V_{q_i}$ to be $\gamma ^i(e_1),\dots , \gamma ^i(e_f)$. One concludes that the data for the matrix of $\gamma$
on the space $\oplus V_{q_i}$ is equivalent to the tuple $(\alpha_1,\dots ,\alpha_f)$. Moreover, the set of all $\ell$th roots of all $\alpha _i$ is the set of eigenvalues of $\gamma$.

Thus the total data for $\gamma$ on $V$ is given by the eigenvalues of $\gamma^{l_i}$ on the space $ V_{q_{i,0}}$ for $i=1,\dots ,r$.
 Since, by assumption, $\gamma$ has determinant 1, the space of possibilities for $\gamma$ is~\smash{$(\mathbb{C}^*) ^{-1+\sum d_i}$}.

The automorphisms of $(V,\{V_q\},\gamma)$ are the $\tau \in {\rm PGL}(V)$ such that $\tau(V_q)=V_q$ for all $q$ and~${\tau \gamma=\gamma \tau}$. One concludes that $\tau$ is determined by its action on all $V_{q_{i,0}}$. Furthermore,
$\tau$~has on this space the same eigenvectors as $\gamma ^{\ell_i}$. This implies that the group of automorphism
is isomorphic to $(\mathbb{C}^*)^{d-1}$. It is seen that the group acts faithfully on the Stokes data. Finally, by scaling
suitable Stokes data to~1, one obtains this affine, open, dense subspace of $\mathcal{R}$.
\end{proof}

 (ii) {\it $z=0$ is regular.} As above in (i), but now with the additional restriction $\gamma \circ {\rm St}_{d_s}\circ \dots \circ {\rm St}_{d_1}=\mathbf{1}_V$. For every candidate, a computation is needed to find out whether $\mathcal{R}$ is not empty and to find its dimension.

 (iii) {\it $z=0$ is irregular singular}. Let $W$ denote the solution space at $z=0$. It has similar additional data as $V$, namely $q$'s,
$\gamma$, ${\rm St}_d$, $\operatorname{mon}_0$. The {\it link} (in \cite{J-M-U} called `connection') is a linear bijection $L\colon W\rightarrow V$ commuting with
the $\operatorname{mon}_*$. The space $\mathcal{R}$ is the space of equivalence classes of the data at $V$, $W$ and the link $L$.

{\it The parameter space $\mathcal{P}$} is defined by data of the topological and the formal monodromies, more precisely
by their characteristic polynomials. By ``fibre'' we will mean a fibre of $\mathcal{R}\rightarrow \mathcal{P}$,
which has the interpretation as space of initial conditions. Each fibre determines a Painlev\'e vector field
(or scalar differential equation) of rank (or order) equal to the dimension of the fibre.

\subsubsection{The rules used for composing our list of families of connections}
The requirements concern the formal data at $z=\infty$ (and also at $z=0$ if this point irregular singular).
\begin{itemize}\itemsep=0pt
\item[R1.] The (distinct) eigenvalues $q_1,\dots , q_r$ with multiplicity $m_1,\dots ,m_r$
 satisfy: all $q_j\neq 0$, $\sum m_j q_j=0$ and (in case $e>1$) invariance under the Galois group of \smash{$\overline{\mathbb{C}(\!(1/z)\!)}$}
over~${\mathbb{C}(\!(1/z)\!)}$.
 Further, one requires that the formal monodromy $\gamma$ is ``generic'', meaning that it has $n$ distinct eigenvalues.

 \item[R2.] If $z=0$ is regular, then the formal data are normalized by the action of the group $\{z\mapsto az+b\}$. If $z=0$ is singular,
 the formal data are normalized using the group~${\{z\mapsto az\}}$. For the description of all formal data at $z=\infty$ and at $z=0$ (if this point is also irregular singular) only one variable $t$ is needed. This is the translation of the requirement that the fibres of $\mathcal{M}\rightarrow \mathcal{R}$ are locally parametrized by a single $t$, called {\it the time variable}.

 \item[R3.] The data should not define a subfamily of a family with more ``time variables''. However, in Section~\ref{thirteen}, we will consider
 a ``companion of $ P_1$'', which is a subfamily of an interesting ``two time variables family''.

 \item[R4.] Apart from individual families, there is interest in hierarchies. By the latter we mean a~sequence of families defined by certain properties of the eigenvalues. For example, $z=0$ regular singular; $e=1$, $\kappa =1$ defines the hierarchy given by the eigenvalues and multiplicities
 $ (z)_{m_1}, (tz)_{m_2}, \bigl( -\frac{m_1+tm_2}{m_3}z\bigr)_{m_3}$ for $m_1,m_2,m_3\geq 1$.
 \end{itemize}

In Section~\ref{list}, a complete list for the cases with multiplicities~1 is presented. This includes of course the classical Painlev\'e equations
with at most two singular points. This list extends in an obvious way to a complete list of hierarchies by allowing multiplicities.

In the next sections, the cases of the list which are not classical, are studied in more detail.

\subsubsection[Details on the definition and construction of the space M]{Details on the definition and construction of the space $\boldsymbol{\mathcal{M}}$}

 (i) \label{(i)} {\it Case $z=0$ is regular singular}. We start by assuming that the irregular singularity $z=\infty$ is unramified and is
given by data $(V, \{V_q\}, \gamma)$.

We choose a basis of $V$, consisting of eigenvectors of $\gamma$, and consider matrices with respect to this basis.
From this, we choose a standard differential operator $z\frac{{\rm d}}{{\rm d}z}+S$ where $S$ is a diagonal matrix with
diagonal entries $( Q_1+a_1,\dots ,Q_n+a_n )$. The $Q_1,\dots ,Q_n$ are the eigenvalues $q_1,\dots ,q_r$, repeated according to their multiplicities. Thus $\sum Q_j=0$. The $a_1,\dots ,a_n \in \mathbb{C}$ satisfy $\sum a_j=0$ and are chosen such that the monodromy of the operator $z\frac{{\rm d}}{{\rm d}z}+\operatorname{diag}(a_1,\dots ,a_n)$ equals $\gamma$.

Now we follow \cite[Section~12]{vdP-Si} and consider the fine moduli space defined by the objects~$(\nabla, \phi)$ on $\mathbb{P}^1$
with the data:
\begin{itemize}\itemsep=0pt
\item[(a)] $\nabla$ is a connection on a trivial vector bundle of rank $n$ on $\mathbb{P}^1$,
\item[(b)] such that $z=0$ is regular singular, and
\item[(c)] $\phi$ is a formal isomorphism of $\nabla$ at $z=\infty$ with $z\frac{{\rm d}}{{\rm d}z}+S$.
\end{itemize}
The explicit choice $\phi$ guarantees that the objects have no automorphisms and that a fine moduli space $\mathcal{U}$
(i.e., a universal family) exists.

 According to \cite[Corollary 12.15 and its proof]{vdP-Si}, the universal family of this fine moduli space is represented by the
family of differential operators $\operatorname{Pr}\bigl(g\bigl(z\frac{{\rm d}}{{\rm d}z}+S\bigr)g^{-1}\bigr)$ where $g$ runs in the $N$-dimensional affine space
\[ \bigg\{{\bf 1}_V+\sum _{k\neq \ell} \operatorname{Hom}(V_k,V_\ell)\otimes _{\mathbb{C}}\bigl(\mathbb{C}z^{-1}+\cdots +\mathbb{C}z^{-\deg (q_k-q_\ell )} \bigr)\bigg\},\]
seen as subset of the group ${\rm SL}_n\bigl(R\bigl[\!\bigl[z^{-1}\bigr]\!\bigr]\bigr)$ with
$R$ the polynomial ring $\mathbb{C}\bigl[\sum _{k\neq \ell}\operatorname{Hom}(V_k,V_{\ell})\bigr]$.
The notation $\operatorname{Pr}$ denotes ``principal part'' and is defined here as
\[\operatorname{Pr}\bigg(z\frac{{\rm d}}{{\rm d}z}+\sum _{k \ll \infty }A_kz^k\bigg)=z\frac{{\rm d}}{{\rm d}z}+\sum _{0 \leq k \ll \infty}A_kz^k.\]
This ends the construction of the universal family $\mathcal{U}$ for the case $z=0$ regular singular and~${z=\infty}$ is unramified.
The group of the automorphisms $G$ of the formal operator $z\frac{{\rm d}}{{\rm d}z}+S$ consists of the diagonal matrices with
 determinant 1, commuting with $S$. This group acts on $\mathcal{U}$ and $\mathcal{M}$ is obtained by dividing $\mathcal{U}$ by the action of $G$.
In other words, $\mathcal{M}$ is obtained from $\mathcal{U}$ by ``forgetting~$\phi$''.

 In general, this categorical quotient has singularities and, moreover need not be the quotient for the set of closed points.
In practise, we will consider a dense affine subspace of $\mathcal{M}$, obtained as closed subspace of $\mathcal{U}$ by
 {\it normalizing} suitable variables to 1 and so providing representatives for the $G$-action.

We note in passing that the above describes a (co-adjoint) orbit of a linear algebraic group over $\mathbb{C}$. Therefore,
$\mathcal{M}$ has a natural symplectic structure, see also~\cite{Bo}.

For the ramified case, one considers the cyclic covering of $\mathbb{P}^1$ of degree $e$, ramified over $0$ and $\infty$.
With respect to the variable $z^{1/e}$, one computes the universal family $z\frac{{\rm d}}{{\rm d}z}+A$ as above, restricted by the
condition of invariance under $\sigma$. The next step
is a computation of the operator on a $\sigma$-invariant basis (compare \cite[Section~12.5]{vdP-Si})) and, finally, dividing by
the action of the group $G$ of automorphisms of $z\frac{{\rm d}}{{\rm d}z}+S$ (by normalizing suitable variables to 1).

\label{(ii)} (ii) {\it The case $z=0$ regular}. From the data $(V,\{V_q\},\gamma)$ one first computes, as above in~(i), a~normalized universal family
of matrix differential operators $z\frac{{\rm d}}{{\rm d}z}+\sum _{0\leq k \ll \infty}A_kz^k$. We propose for $\mathcal{M}$ the subfamily
defined by the condition that all the entries of $A_0$ are zero. An explicit computation is needed to verify
whether $\mathcal{M}$ is not empty and to compute its dimension.

 (iii) {\it The case $z=0$ and $z=\infty$ irregular singular}. One expects a universal family of differential operators
\[
z\frac{{\rm d}}{{\rm d}z}+\sum _{-\infty \ll k \ll \infty}A_kz^k.
\] For the right-hand part \smash{$\sum _{0\leq k \ll \infty} A_kz^k$}, the method of~(i) produces
a proposal. The same holds for the left-hand part $\sum _{-\infty \ll k \leq 0}A_kz^k$. Gluing of the two proposals may result in a suitable family.
A priori, it is not clear whether the formal data at $z=0$ and at $z=\infty$ can be combined to a family $\mathcal{M}$ and a
corresponding monodromy space~$\mathcal{R}$.

 {\it Comments.}
The explicit computation of $\mathcal{R}$ works quite well. The computation of $\mathcal{M}$ in cases~(i) or~(ii) may fail or
may lead to a result unsuitable for further analysis,
due to complexity. For case (iii), one needs a good guess to start the computation.

 The number of cases where
a complete computation of the Lax pairs can be given is, again due to complexity, rather small. In case (i), the differential
equations involve \[N=\sum _{k,\ell} \deg(q_k-q_\ell)\cdot \dim V_{q_k}\cdot \dim V_{q_\ell}\] functions of $t$. This system is mostly too
large.

A first step towards simplification is normalization by scaling the basis vectors of $V$ (i.e., the step from $\mathcal{U}$ to $\mathcal{M}$).
A next step is to reduce this system by the use of invariants, which are independent of $t$ (i.e., considering fibres of
$\mathcal{M}\rightarrow \mathcal{P}^+$ and $\mathcal{R}\rightarrow \mathcal{P}$). This reduces the number of functions of $t$ involved
in the Lax pair equations but can be a source of complexity.

\section{List of all cases with one time variable}\label{list}

 The computation of this list is straightforward, but somewhat long.

{\bf With $\boldsymbol{z=0}$ regular singular or regular.}
The condition ``one time variable $t$'' implies that there are at most three Galois orbits of eigenvalues.
A regular singular case can restrict to a regular case, e.g., $z$, $tz$, $(-1-t)z$ and $z=0$ regular exists and produces trivial Stokes data. In the table representatives for the Galois orbits of the eigenvalues are given; the rightmost column indicates
in which section this example is discussed and which classical Painlev\'{e} equation $P_j$ it is related to. As in Section~\ref{FlNe}, the Flaschka--Newell
equation which is a deformation of $P_2$, is denoted $P_{2fn}$ (see \cite[Section~4.9]{vdP-Sa} for this case).
\begin{itemize}\itemsep=0pt
\item One Galois orbit, $z=0$ regular singular,
\begin{alignat*}{3}
 & z^{3/2}+tz^{1/2}, && (P_{2fn}) &&\\
 & z^{2/e}+tz^{1/e} \quad \text{for} \quad e\geq 3 .\qquad && (\text{Section~\ref{two}},\, P_4,\,P_5)&&
\end{alignat*}
\item One Galois orbit, $ z=0 $ regular,
\begin{alignat*}{3}
 & z^{5/2}+tz^{1/2}, \qquad&& (P_1) && \\
 & z^{4/3}+tz^{2/3} . && (\text{Section~\ref{three}})&&
 \end{alignat*}
\item Two Galois orbits, $ z=0 $ regular singular,
\begin{alignat*}{3}
 & z^2+tz , \ -\bigl(z^2+tz\bigr), && (P_4) \\
 & z+tz^{1/2} ,\ -2z . && (\text{Section~\ref{four}}) \\
 & z^{1/e_1} , tz^{1/e_2}\quad \text{for} \quad e_1\geq e_2\geq 2 .\qquad&& (\text{Section~\ref{five}}, \, P_6)
 \end{alignat*}
\item Two Galois orbits, $ z=0 $ regular,
\begin{alignat*}{3}
& z^3+tz ,\ -\bigl(z^3+tz\bigr) .\qquad && (\text{Section \ref{six}},\, P_2)
\end{alignat*}
\item Three Galois orbits, $ z=0 $ regular singular,
\begin{alignat*}{3}
 & z ,\ tz , \ (-1-t)z , \qquad&& (\text{Section \ref{seven}},\, P_6) \\
& z^{1/e} ,\ tz , \ -tz \quad \text{with}\quad e>1 .\qquad&& (\text{Section~\ref{sec9}}) \end{alignat*}
\item Three Galois orbits, $ z=0 $ regular,
\begin{alignat*}{3}
&z^2 , \ -z^2-tz , \ tz . \qquad&& (\text{Section \ref{nine}},\, P_4)
\end{alignat*}
\end{itemize}

{\bf With both $\boldsymbol{z=0}$ and $\boldsymbol{z=\infty}$ irregular singular.}
\begin{itemize}\itemsep=0pt
\item $1/z,-1/z$ at $z=0$ and $tz$, $-tz$ at $z=\infty$ (Section~\ref{sec11}, $P_3(D_6)$).
\item $z^{-1/2}$ at $z=0$ and $tz$, $-tz$ at $z=\infty$ (Section~\ref{sec12}, $P_3(D_7)$).
\item $z^{-1/n}$ at $z=0$ and $tz^{1/n}$ at $z=\infty$ with $n\geq 2$ (Section~\ref{twelve},\, $P_3(D_8)$)
\end{itemize}

More general: $n_1,n_2\geq 2$ and $z^{-1/n_1}$ at $z=0$ and $tz^{1/n_2}$ at $z=\infty$
 (with suitable multiplicities).

We discuss these cases in the indicated sections.

\section[n-th root of z squared plus t times n-th root of z, hierarchy of Noumi and Yamada]{$\boldsymbol{z^{2/n}+tz^{1/n}}$, $\boldsymbol{n\geq 3}$, hierarchy of Noumi and Yamada}\label{two}
We study the structure of the moduli spaces $\mathcal{M}_n$, $\mathcal{R}_n$ and the Lax pair
computations, separately for $n$ odd and $n$ even.

\subsection[The moduli spaces M\_n and R\_n for odd n]{The moduli spaces $\boldsymbol{\mathcal{M}_n}$ and $\boldsymbol{\mathcal{R}_n}$ for odd $\boldsymbol{n}$}\label{2.1}

 {\it Computations for $\mathcal{R}_n$}.
 A module $M\in \mathcal{M}_n$ has the eigenvalues $q_j=\sigma ^j(q_0)=\omega^{2j}z^{2/n}+t\omega ^jz^{1/n}$ for $j=0,\dots ,n-1$
at $z=\infty$, where $\omega:={\rm e}^{2\pi {\rm i}/n}$.

The tuple $(V,\{V_q\},\gamma,\{{\rm St}_d\})$ that classifies $M$ at $z=\infty$
has the form $V=\mathbb{C}e_0\oplus \cdots \oplus \mathbb{C}e_{n-1}$ where $\mathbb{C}e_j=V_{q_j}$ for
$j=0,\dots , n-1$. This basis is chosen such that
 $\gamma$ satisfies $e_0\mapsto e_1\mapsto \cdots \mapsto e_{n-1}\mapsto e_0$.

 The space of the Stokes matrices at $z=\infty$ is isomorphic to $\mathbb{C}^{N}$, where $N=n(n-1)\cdot \frac{2}{n}=2(n-1)$.
 Since the basis vectors $e_0,\dots , e_{n-1}$ of $V$ are unique up to multiplication by the same constant, one finds $\mathcal{R}_n=\mathbb{C}^{2(n-1)}$ (see Lemma~\ref{lemma01}) and $\dim \mathcal{M}_n=1+\dim \mathcal{R}_n=1+2(n-1)$.

For a module $M\in \mathcal{M}_n$, the data of the topological monodromy $\operatorname{mon}_0$ at $z=0$ is the conjugacy class of a matrix in ${\rm SL}_n$.
A conjugacy class is mapped to its characteristic polynomial~${T^n+a_{n-1}T^{n-1}+\cdots +a_1T+(-1)^n}$ and the parameter space
$\mathcal{P}_n$ is the space of all possible characteristic polynomials and thus isomorphic to $\mathbb{C}^{n-1}$.
The ``monodromy identity'' and a~nontrivial, explicit computation, similar to the ones in \cite{CM-vdP}, shows:

{\it $\mathcal{R}_n\rightarrow \mathcal{P}_n$ is surjective and the fibres have dimension $(n-1)$.}
 The fibre for $n=3$ is computed in \cite{vdP-T} to be the affine surface $xyz+x^2+p_1x+p_2y+p_3z=p_4$ for certain constants $p_i$.
 This is expected because the computation of the Lax pair equation leads to an identification with $P_4$.
 A further computation shows that $\operatorname{mon}_0={\bf 1}$ is not possible, although characteristic polynomial~${(T-1)^3}$ is possible.

 For $n=5$, we present details of the computation the fibre. Write
 $\omega ={\rm e}^{2\pi {\rm i}/5}$, $q_0=z^{2/5}+tz^{1/5}$,
 $q_1=\omega^2z^{2/5}+\omega t z^{1/5}, \dots , q_4=\omega^3z^{2/5}+\omega^4tz^{1/5}$.
 The singular directions in $[0,1)$ are $7/8$ for~${q_0-q_2}$, $ q_3-q_4$, $5/8$ for $q_1-q_4$, $ q_3-q_2$, $3/8$ for $q_1-q_2$, $ q_3-q_0$ and
 $1/8$ for $q_1-q_0$, $ q_4-q_2$. The fibres of $\mathcal{R}_5\rightarrow \mathcal{P}_5$ are rational 4-folds. After eliminating 3 of the 8
 variables for $\mathcal{R}_5$, the affine fibre is given by a degree 5 polynomial equation in 5 variables and with 4 parameters.
This computation also shows that $\operatorname{mon}_0={\bf 1}$ is not possible for $n=5$.

{\it Construction of $\mathcal{M}_n$ and the Lax pairs}.
We make the method explained in Section~\ref{section0} explicit. A differential module
$M\in \mathcal{M}_n$ over $\mathbb{C}(z)$ is replaced by $N:=\mathbb{C}\bigl(z^{1/n}\bigr)\otimes M$. Let $D$ denote the differential operator
\smash{$\nabla_{z\frac{{\rm d}}{{\rm d}z}}$} on $M$. Now $D$ extends uniquely to a~differential operator, also called~$D$, on~$N$. This $D$ commutes
with the semi-linear automorphism~${\sigma\colon N\rightarrow N}$, induced by the automorphism $\sigma$ of $\mathbb{C}\bigl(z^{1/n}\bigr)$, given by
$\sigma z^{1/n}=\omega z^{1/n}$. Thus the $M\in \mathcal{M}$ are replaced by pairs $(N,\sigma)$, as above.

 Let $e_0,\dots ,e_{n-1}$ be a basis of $N$ over $\mathbb{C}\bigl(z^{1/n}\bigr)$ such that the map $\sigma$
 satisfies $\sigma \colon e_0\mapsto e_1\mapsto \cdots \mapsto e_{n-1}\mapsto e_0$. The operator $D$ is determined by
 $D(e_0)$. The formula
 \[
 D(e_0)=\bigl(z^{2/n}+tz^{1/n}\bigr)e_0+\sum _{i=1}^{n-1}\bigl(a_i+b_iz^{1/n}\bigr)e_i
 \]
 is supported by \cite[Sections~12.3--12.5]{vdP-Si} (compare Section~\ref{section0}).
 For the operator $E:=\frac{{\rm d}}{{\rm d}t}+B$ such that $\{D, E\}$ forms a Lax pair, one can verify the assumption that $E$
is the $\sigma$-invariant operator with \smash{$E(e_0)=z^{1/n}e_0+\sum_{j=1}^{n-1} c_je_j$}.

 One deduces from this the matrix of $D$ with respect to the basis $B_0,\dots ,B_{n-1}$ of $M:=N^{\langle\sigma\rangle}$,
 where $B_j:=\sum_{k=0}^{n-1} \sigma^k\bigl(z^{j/n}e_0\bigr)$ for $0\leq j\leq n-1$ and $B_n:=zB_0$, $
 B_{n+1}:=zB_1$. The formula is
 \[D(B_j)=\frac{j}{n}B_j+\sum _{i=1}^{n-1}a_i\omega^{-ij}B_j+tB_{j+1}+\sum_{i=1}^{n-1}b_i\omega^{-i(j+1)}B_{j+1}+B_{j+2}.\]
 The formula for $E$ on this basis is
 \[
 E(B_j)= B_{j+1}+\left(\sum _{k=1}^{n-1}\omega^{-kj}c_k\right)B_j.
 \]
 Put $\epsilon_j=\frac{j}{n}+\sum _{i=1}^{n-1}a_i\omega^{-ij}$ and $f_j=t+\sum_{i=1}^{n-1} b_i\omega^{-ij}$. The operator $D$ is
\[
 z\frac{{\rm d}}{{\rm d}z} + \left(\begin{matrix}
 \epsilon_0 &0 &0 &* &* &z & zf_0\\
 f_1 & \epsilon_1& 0 &0 &* & 0 & z \\
 1 & f_2 & \epsilon_2 & 0 & * & * & 0 \\
 0 & 1 & f_3 &\epsilon_3 & 0 & * & * \\
 * & * & * & * & * & * & * \\
 * & * & * & 1 & f_{n-2} &\epsilon_{n-2} & 0 \\
 * & * & * & 0 & 1 & f_{n-1} & \epsilon_{n-1}
 \end{matrix}\right),
 \]
 note that $\sum \epsilon _j=\frac{n-1}{2}$ and $\sum f_j=nt$.
The $\epsilon_0,\dots ,\epsilon_{n-1}$ are the parameters of the family.
 The~$\bigl\{{\rm e}^{2\pi {\rm i} \epsilon_j}\bigr\}$ are the eigenvalues of the topological monodromy at $z=0$. These can be seen as parameters for $\mathcal{R}_n$. For an isomonodromic family the $\epsilon_j$ are constant and the
$f_0,\dots ,f_{n-1}$ are analytic functions of the parameter $t$.

 The operator $E$ reads on the above basis
\[\frac{{\rm d}}{{\rm d}t}+ \left(\begin{matrix}
 g_0 &0 &0 &0 &* &* & z\\
 1 & g_1& 0 &0 &* & * & 0 \\
 0 & 1 & g_2 & 0 & * & * & 0 \\
 * & * & * & * & * & * & * \\
 * & * & * & * &* & * & * \\
 0 & 0 & 0 & 0 & 1 &g_{n-2} & 0 \\
 0 & 0 & 0 & 0 & 0 & 1 & g_{n-1}
 \end{matrix}\right)
 \]
with $g_j=\bigl(\sum _{k=1}^{n-1}\omega^{-kj}c_k\bigr)$ and $\sum g_j=0$.
For an isomonodromic family, the $\{g_j\}$ are functions of $t$ and are in fact eliminated by the Lax pair
condition $DE=ED$.

For $n=5$, the Painlev\'e type differential system for this Lax pair is 
 \begin{gather*}
 f_1'=f_1(-f_1-2f_2-2f_4+t)+2\epsilon_1+\epsilon_2+\epsilon_3+\epsilon_4,\\
f_2'=f_2(-2f_1+f_2-2f_4-t)-\epsilon_1+\epsilon_2,\qquad
 f_3'=f_3(-2f_1-f_3-2f_4+t)-\epsilon_2+\epsilon_3,\\
 f_4'=f_4(2f_1+2f_3+f_4-t)-\epsilon_3+\epsilon_4.\end{gather*}

For $n=7$, one has $ \sum f_j=7t$, $ \sum \epsilon_j=3$ and
 \begin{gather*}
 f'_0 = f_0(-f_1+f_2-f_3+f_4-f_5+f_6)+ \epsilon_0 - \epsilon_6 + 1,\\
f'_1 =f_1(f_0-f_2+f_3-f_4+f_5-f_6) - \epsilon_0 + \epsilon_1,\\
f'_2 =f_2( -f_0+f_1 -f_3+f_4-f_5+f_6)- \epsilon_1 + \epsilon_2,\\
f'_3 = f_3(f_0-f_1+f_2-f_4+f_5-f_6) - \epsilon_2 + \epsilon_3,\\
f'_4 =f_4(-f_0+f_1-f_2+f_3-f_5+f_6) - \epsilon_3 + \epsilon_4,\\
f'_5 = f_5(f_0-f_1+f_2-f_4+f_4-f_6) - \epsilon_4 + \epsilon_5,\\
f'_6 =f_6(-f_0+f_1-f_2+f_3-f_4+f_5) - \epsilon_5 + \epsilon_6.\end{gather*}
The general case for odd $n$ is similar.

 {\it Observation}. Apart from small changes the above is the symmetric Lax pair
 introduced by Noumi, Yamada et al.\ (see \cite{N-Y,S-H-C}). The changes are
 \begin{itemize}\itemsep=0pt
\item[(a)] $\epsilon_j$ is changed into $\epsilon_j-\frac{n-1}{2n}$ in order to obtain a matrix with trace zero. This corresponds to a small change in the definition of $D$, namely
 \[D(e_0)=\left(z^{2/n}+tz^{1/n}-\frac{n-1}{2n}\right) e_0+\sum _{i=1}^{n-1}\bigl(a_i+b_iz^{1/n}\bigr) e_i.\]
\item[(b)] A notational change of $t$ into $\frac{t}{n}$.

\item[(c)] Transposing the matrix. This is due to the relation between a covariant solution space and a contravariant solution space.
\end{itemize}

 {\it An alternative method for $n=3$, i.e., the Noumi--Yamada form for $P_4$.}
The Lax pair equations are equivalent to $ED(e_0)=DE(e_0)$. The
normalized operator $D$ given as
\[De_0=\left(z^{2/3}+tz^{1/3}+\frac{2}{3}\right)e_0+(a_1+b_1z^{1/3})e_1+\bigl(a_2+b_2z^{1/3}\bigr)e_2,\] where
$a_1$, $a_2$ are constants and $b_1$, $b_2$ are functions of $t$, commutes with the operator $ E$ such that
$E(e_0)=z^{1/3}e_0+c_1e_1+c_2e_2$ (for suitable functions $c_1$, $c_2$ of $t$) if and only $b_1$, $b_2$ satisfy the differential equations
\begin{gather*}
\begin{split}
& b_1'=a_1(1-\omega)+b_1t(2\omega +1)+b_2^2(-2\omega -1),\\
& b_2'=a_2(\omega +2)+b_1^2(2\omega +1) +b_2t(-2\omega -1).
\end{split}
\end{gather*}
This is in fact a Hamiltonian system $b_1'=\frac{\partial H}{\partial b_2}$, $b_2'=-\frac{\partial H}{\partial b_1}$
with $\omega={\rm e}^{2\pi {\rm i}/3}$ and
\[
H=-\left(\frac{b_1^3}{3}+\frac{b_2^3}{3}\right)(2\omega +1)+b_1b_2t(2\omega +1)-b_1a_2(\omega +2) -b_2a_1(\omega -1).\]

After a linear change of variables this Hamiltonian coincides with Okamoto's standard Hamiltonian for $P_4$
(see \cite[p.~265]{O1}).

\subsection[The moduli spaces M\_n and R\_n for even n]{The moduli spaces $\boldsymbol{\mathcal{M}_n}$ and $\boldsymbol{\mathcal{R}_n}$ for even $\boldsymbol{n}$}

We proceed as in Section~\ref{2.1}.
Write $n=2m$ and $\omega ={\rm e}^{2\pi {\rm i}/n}$. The eigenvalues at $z=\infty$ are
$q_j= \omega^{j/m}z^{1/m}+ \omega ^{j/2m}tz^{1/2m}$ for $j=0,\dots ,2m-1$.

Now $N:=\sum _{i\neq j} \deg (q_i-q_j)=4m-3$, $\mathcal{R}_{2m}\cong \mathbb{C}^{N}$ and $\mathcal{M}_{2m}$
has dimension $1+4m-3$. The parameter space $\mathcal{P}_{2m}$ for the monodromy space
consists of the characteristic polynomials of the monodromy at $z=0$. Since $\Lambda^{2m}M$ is the trivial differential module, this
monodromy has determinant 1. Thus $\dim \mathcal{P}_{2m}=2m-1$. A similar explicit computation as the one
mentioned in Section~\ref{2.1} shows:
{\it $\mathcal{R}_{2m}\rightarrow \mathcal{P}_{2m}$ is surjective and the fibres have dimension $2m-2$}.
We make the method of construction a differential operator, a Lax pair and Painlev\'e type equations explicit for $n=4$. The general case is discussed after that.

\subsubsection[The case n=4]{The case $\boldsymbol{n=4}$}\label{section2.2.1}
Following Section~\ref{section0}, we may assume that the differential operator $D$ has on the basis $e_0$, $e_1$, $e_2$, $e_3$ the
 formula
 \[D(e_0)=\left(z^{1/2}+\frac{t}{4} z^{1/4}-3/8\right) e_0 +\bigl(a_1+b_1z^{1/4}\bigr) e_1+a_2 e_2+\bigl(a_3+b_3z^{1/4}\bigr) e_3, \]
 and $D$ commutes with $\sigma$ defined by $\sigma e_j=e_{j+1}$ for $j=0,1,2$ and $\sigma e_3=e_0$ and
 $\sigma z^{1/4}={\rm i}z^{1/4}$.
 Consider the following basis of invariants:
 \begin{alignat*}{3}
& B_0=e_0+e_1+e_2+e_3,\qquad && B_1=z^{1/4}\bigl(e_0+{\rm i}e_1+{\rm i}^2e_2+{\rm i}^3e_3\bigr),& \\
& B_2=z^{1/2}(e_0-e_1+e_2-e_3),\qquad && B_3=z^{3/4}(e_0-{\rm i}e_1-e_2+{\rm i}e_3).&
\end{alignat*}
 The matrix of $D$ with respect to this basis is
\[ \begin{pmatrix} -\frac{3}{8}+a_1+a_2+a_3 & 0 & z & z\bigl(\frac{t}{4}+b_1+b_3\bigr) \\
 \frac{t}{4}-{\rm i}b_1+{\rm i}b_3 & -\frac{1}{8}-{\rm i}a_1-a_2+{\rm i}a_3 & 0 &z \\
 1 &\frac{t}{4}-b_1-b_3 &\frac{1}{8}-a_1+a_2-a_3 & 0 \\
 0 & 1 & \frac{t}{4}+{\rm i}b_1-{\rm i}b_3 &\frac{3}{8}+{\rm i}a_1-a_2-{\rm i}a_3
 \end{pmatrix} \]
 and $D$ is equal to the differential operator
 \[
 z\frac{{\rm d}}{{\rm d}z}+\begin{pmatrix} \epsilon_0&0&z&zf_0\\
 f_1&\epsilon_1&0&z\\ 1&f_2&\epsilon_2& 0\\ 0&1&f_3&\epsilon_3 \end{pmatrix}
 \]
 with $\sum \epsilon _j=0$, $f_0+f_2=f_1+f_3=\frac{t}{2}$. The $\epsilon_0,\dots ,\epsilon_3$ are parameters.

 The operator $D$ is completed to a Lax pair by the differential operator $E$ with respect to~$\frac{{\rm d}}{{\rm d}t}$. This
 operator, written on the basis $e_0$, $e_1$, $e_2$, $e_3$ is $\sigma$-invariant and has the form
 $E(e_0)=z^{1/4}e_0+\sum _{j=1}^3h_je_j$ for suitable functions $h_1$, $h_2$, $h_3$ of $t$. On the basis
 $B_0 $, $B_1$, $B_2$ one obtains
 \[
 E:=\frac{{\rm d}}{{\rm d}t}+ \begin{pmatrix}g_0& 0& 0&z\\ 1&g_1&0&0\\0&1&g_2&0\\0&0&1&g_3 \end{pmatrix}
 \]
 with $\sum g_j=0$. The assumption that $E$ commutes with $D$ produces equations for
 the derivatives of $f_0$, $f_1$, $f_2$, $f_3$, seen as functions of $t$. These formulas are similar to those derived by
 Noumi--Yamada. Moreover, combining the differential equations for $f_0$ and $f_1$ leads to the standard $P_5$ equation,
 see \cite{N-Y,S-H-C} for details.

{\it An alternative computation} is a consequence of the observation that isomonodromy is given by
 $ DE(e_0)=ED(e_0)$ and $a_1,a_2,a_3 \in \mathbb{C}$. This produces equations with
 parameters $a_1$, $a_2$, $a_3$
 \begin{gather*}
 4t\cdot \frac{{\rm d}b_1}{{\rm d}t} = -16{\rm i}b_1^2b_3 + t^2b_1{\rm i} + 16{\rm i}b_3^3 - 4{\rm i}a_1t + 4a_1t - 32a_2b_3\\
 4t\cdot \frac{db_3}{dt} = -16{\rm i}b_1^3 + 16{\rm i}b_1b_3^2 - t^2b_3{\rm i} + 4{\rm i}a_3t - 32a_2b_1 + 4a_3t,\\
 h_1 = -{\rm i}b_1/2 + b_1/2, \qquad t\cdot h_2 = 2b_1^2{\rm i} - 2b_3^2{\rm i} + 4a_2,\qquad h_3 = b_3{\rm i}/2 + b_3/2. \end{gather*}
 One observes that the equations for $b_1$, $b_3$ form a Hamiltonian system with
 \begin{gather*}
 t\frac{{\rm d}b_1}{{\rm d}t}=\frac{\partial H}{\partial b_3},\qquad t\frac{{\rm d}b_3}{{\rm d}t}=-\frac{\partial H }{\partial b_1},\\
H=\frac{{\rm i}\bigl(b_1^2-b_3^2\bigr)^2}{t}+\frac{{\rm i}tb_1b_3}{4}+\frac{4a_2\bigl(b_1^2-b_3^2\bigr)}{t}-(1+{\rm i})a_3b_1+(1-{\rm i})a_1b_3.\end{gather*}

{\it Comments}.
The above Hamiltonian $H$ and the differential equations for $b_1$, $b_3$ coincide, after a linear change of variables,
with Okamoto's standard polynomial Hamiltonian for $P_5$, see \cite[p.~265]{O1}.

{\it The fibers of $\mathcal{R}_4\rightarrow \mathcal{P}_4$}. The eigenvalues at $z=\infty$ are:
$q_0=z^{1/2}+tz^{1/4}$, $q_1=-z^{1/2}+{\rm i}tz^{1/4}$, $q_2=z^{1/2}-tz^{1/4}$, $q_3=-z^{1/2}-{\rm i}tz^{1/4}$.
 The differences $q_0-q_1$, $ q_0-q_3$, $ q_2-q_1$, $ q_2-q_3$ have the form~${2z^{1/2}+\cdots}$ and further
 $q_0-q_2=2tz^{1/4}$ and $q_1-q_3=2{\rm i}tz^{1/4}$. There is one singular direction in $[0,1)$ for the terms
 $\pm 2z^{1/2}$. For the terms $\pm 2tz^{1/4}$, $\pm 2{\rm i}tz^{1/4}$ there is only one singular direction in $[0,1)$.
 For a suitable choice of $t$, this leads to the monodromy identity
 \[\operatorname{mon}_\infty= \begin{pmatrix} & & &-1 \\ 1& & & \\ & 1& & \\ & &1 & \end{pmatrix} \begin{pmatrix} 1 & & & \\ & 1& &y \\ & &1 & \\ & & & 1\end{pmatrix}
\begin{pmatrix} 1 & & & \\ x_1&1 &x_2 & \\ & & 1& \\ x_3& &x_4 & 1\end{pmatrix}. \]

 One observes that $\operatorname{mon}_0=\operatorname{mon}_\infty^{-1}$ cannot be the identity. Thus for the data considered here, $z=\infty$ singular with
 eigenvalues $z^{1/2}+tz^{1/4}$ and its conjugates and $z=0$ {\it regular} cannot be combined.

 The space $\mathcal{P}_4$ is parametrized by $p_1$, $p_2$, $p_3$, where the characteristic polynomial of $\operatorname{mon}_\infty$ is written as $T^4+p_3T^3+p_2T^2+p_1T+1$. For a suitable choice of
 elimination of two variables (e.g., $x_1$, $x_2$), the fibres are described by a cubic equation in three variables $y$, $x_3$, $x_4$ and parameters~$p_1$,~$p_2$,~$p_3$. The equation reads
 $v_1v_2v_3+*v_1^2+*v_2^2+*v_1+*v_2+*v_3+*=0$ for suitable affine expressions $*$'s in the parameters
 $p_1$, $p_2$, $p_3$. This is expected, since the related Painlev\'e equation $P_5$ has the same cubic equation for its monodromy.

 \subsubsection[The general case with n=2m]{The general case with $\boldsymbol{n=2m}$}

 Consider the $\mathbb{C}(t)\bigl[z^{1/2m}\bigr]$-lattice with basis $e_0,\dots ,e_{2m-1}$, provided with
 the action of $\sigma$ given by the formulas: $\sigma z^\lambda ={\rm e}^{2\pi {\rm i} \lambda}z^\lambda$ and
 $\sigma e_j=e_{j+1}$ for $j=0,\dots ,2m-2$ and $\sigma e_{2m-1}=e_0$. The operator $D$
 (with respect to the derivation $z\frac{{\rm d}}{{\rm d}z}$) representing $\mathcal{M}_{2m}$, is $\sigma$-invariant
 and is given~by
 \[D(e_0)=\left(z^{1/m}+\frac{t}{2m}z^{1/2m}-\frac{2m-1}{4m}\right)e_0 +\sum_{j=1}^{2m-1}\bigl(a_j+b_jz^{1/2m}\bigr)e_j,\]
with varying $a_j,b_j\in \mathbb{C}$ and where $b_m=0$.
 On the basis $B_0,\dots, B_{2m-1}$ of invariants, $D$ has the form $z\frac{{\rm d}}{{\rm d}z}+A_0+zA_1$.
 We make this explicit for $n=6$, $m=3$. The general $n=2m$ case is similar,
 \[ D=z\frac{{\rm d}}{{\rm d}z}+
\begin{pmatrix} \epsilon_0 &0 &0 &0 &z &z\bigl(\frac{t}{6}+f_0\bigr) \\
 \frac{t}{6}+f_1&\epsilon_1 &0 &0 &0 &z \\
 1&\frac{t}{6}+f_2 &\epsilon_2 &0 &0 &0 \\
 0& 1& \frac{t}{6}+f_3&\epsilon_3 &0 &0 \\
 0&0 &1 &\frac{t}{6}+f_4 &\epsilon_4 &0\\
 0&0 &0 & 1&\frac{t}{6}+f_5 & \epsilon _5\end{pmatrix}. \]
 The $\epsilon_j$ are linear combinations of $a_1,\dots ,a_5$ satisfying $\sum \epsilon_j=0$, and the
 $f_0,\dots ,f_5$ are linear combinations of $b_1$, $b_2$, $b_4$, $b_5$
 such that the relations $f_0+f_2+f_4=f_1+f_3+f_5=0$ hold.

 One observes that the data of the eigenvalues of $A_0$ are equivalent to $a_1,\dots , a_{2m-1}$. Thus for an
 isomonodromic family the $a_j$ are constant and the $b_1,\dots, b_{2m-1}$ (with the condition~${b_m=0}$)
 are functions~of~$t$. The differential equations for the $b_j$ are derived from a
 Lax pair $z\frac{{\rm d}}{{\rm d}z}+A_0+zA_1$, $ \frac{{\rm d}}{{\rm d}t}+B$ with an, a priori, unknown matrix $B$ depending on $t$ and $z$.
 The action of the operator~${\frac{{\rm d}}{{\rm d}t}+B}$ on the $\mathbb{C}(t)\bigl[z^{1/2m}\bigr]$-lattice with basis $e_0,\dots ,e_{2m-1}$ is called $E$. It is $\sigma$-invariant and one can prove that $E$ is given by
 \[ E(e_0)=z^{1/2m}e_0+\sum _{j=1}^{2m-1}h_je_j \qquad \mbox{for certain functions } h_1,\dots ,h_{2m-1} \mbox{ of } t.\]
 We make this explicit for $n=6$, $m=3$ (again, the general case $n=2m$ is similar),
 \[ E=\frac{{\rm d}}{{\rm d}t}+
\begin{pmatrix} g_0 &0 &0 &0 &0 &z \\
 1&g_1 &0 &0 &0 &0 \\
 0&1 &g_2 &0 &0 &0 \\
 0& 0& 1&g_3 &0 &0 \\
 0&0 &0&1 &g_4 &0 \\
 0&0 &0 & 0&1 &g _5\end{pmatrix}\qquad\mbox{and} \qquad\sum g_j=0. \]
The Painlev\'e type equations are similar to those in Section~\ref{section2.2.1}.

 {\it An alternative computation}. From the equation $DE(e_0)=ED(e_0)$ and all $a_j$ are constants, the differential equations for $b_1,\dots ,b_{2m-1}$ follow. For the case $n=6$, $m=3$, one obtains the
 following Painlev\'e type equations
 for $b_1$, $b_2$, $b_4$, $b_5$ ($b_3=0$, $w={\rm e}^{2\pi {\rm i}/6}$)
\begin{align*}
-18t\frac{{\rm d}b_1}{{\rm d}t} ={}& 288b_1b_2b_4w - 2b_1t^2w - 24b_2b_5tw - 288b_4^2b_5w + 18a_1tw \\
& - 144b_1b_2b_4
+ b_1t^2 + 12b_2b_5t + 144b_4^2b_5 - 18a_1t + 216a_3b_4,\\
- 6t\frac{{\rm d}b_2}{{\rm d}t} ={}& -4b_1^2tw + 96b_1b_2b_5w - 2b_2t^2w + 12b_4^2tw - 96b_4b_5^2w + 6a_2tw \\
& + 2b_1^2t - 48b_1b_2b_5 + b_2t^2 - 6b_4^2t + 48b_4b_5^2 - 12a_2t + 72a_3b_5, \\
 -6t\frac{{\rm d}b_4}{{\rm d}t} ={}& 96b_1^2b_2w - 96b_1b_4b_5w - 12b_2^2tw +
 2b_4t^2w + 4b_5^2tw - 6a_4tw \\ & - 48b_1^2b_2 + 48b_1b_4b_5 + 6b_2^2t
 - b_4t^2 - 2b_5^2t + 72a_3b_1 - 6a_4t,\\
-18t \frac{{\rm d}b_5}{{\rm d}t} ={}& 88b_1b_2^2w + 24b_1b_4tw - 288b_2b_4b_5w + 2b_5t^2w
 - 18a_5tw \\
 & - 144b_1b_2^2 - 12b_1b_4t + 144b_2b_4b_5
 - b_5t^2 + 216a_3b_2.
 \end{align*}
There is a Hamiltonian function $H$ such that
\[ \frac{{\rm d}b_5}{{\rm d}t}=-\frac{\partial H}{\partial b_1},\qquad \frac{{\rm d}b_1}{{\rm d}t}=\frac{\partial H}{\partial b_5},\qquad \frac{{\rm d}b_4}{{\rm d}t}=-\frac{\partial H}{\partial b_2},\qquad
\frac{{\rm d}b_2}{{\rm d}t}=\frac{\partial H}{\partial b_4},\]
where $H$ is defined by
\begin{gather*}
3tH= (w - 1/2)(b_2b_4 + b_1b_5/3)t^2 + \bigl((-2w + 1)b_2^3 + \bigl(\bigl(2b_5^2 - 3a_4\bigr)w - b_5^2 - 3a_4\bigr)b_2 \\
 \phantom{3tH=}{} + (1-2w)b_4^3
 + \bigl(\bigl(2b_1^2 - 3a_2\bigr)w - b_1^2 + 6a_2\bigr)b_4 - (3a_1b_5 + 3a_5b_1)w + 3a_1b_5\bigr)t \\
 \phantom{3tH=}{}+ 12(-b_1b_2 + b_4b_5)((-b_1(2w - 1)b_2 + b_5(2w - 1)b_4 - 3a_3),
 \end{gather*}
 with $w={\rm e}^{2\pi {\rm i} /6}$. This is related to the equations studied in \cite{H-K-N-S}.

\section[z-power(4/3)+t*z-power(2/3) and regular z=0]{$\boldsymbol{z^{4/3}+tz^{2/3}}$ and regular $\boldsymbol{z=0}$}\label{three}
The $q_0=z^{4/3}+tz^{2/3}$, $q_1=\omega z^{4/3}+\omega ^2tz^{2/3},
q_2=\omega^2z^{4/3}+\omega tz^{2/3}$ with $\omega={\rm e}^{2\pi {\rm i} /3}$ are the eigenvalues. First, we assume that $z=0$ is a~regular singular point. Then, the monodromy
space~$\mathcal{R}$ is isomorphic to $\mathbb{C}^8$. The monodromy $\operatorname{mon}$ at $z=\infty$ (or equivalently at $z=0$) is a~product of the formal monodromy and 8 Stokes matrices. The singular directions in~$[0,1)$ are
\smash{$\frac{15}{16}$}, \smash{$\frac{13}{16}$}, \smash{$\frac{11}{16}$}, $\frac{9}{16}$, $\frac{7}{16}$, $\frac{5}{16}$, $\frac{3}{16}$, $\frac{1}{16}$
for $q_1-q_2$, $q_1-q_0$, $q_2-q_0$, $q_2-q_1$, $q_0-q_1$, $q_0-q_2$, $q_1-q_2$, $q_1-q_0$.
Each Stokes matrix has one nontrivial entry and these are in the same order
$x_{12}$, $x_{10}$, $x_{20}$, $x_{21}$, $x_{01}$, $x_{02}$, $y_{12}$, $y_{10}$.
A computation shows that the map $\mathcal{R}\rightarrow {\rm SL}_3(\mathbb{C})$, which send the Stokes data to $\operatorname{mon}$,
is birational. Moreover, the preimage of ${\bf 1} \in {\rm SL}_3(\mathbb{C})$ is one point, namely
$x_{01} = -1$, $ x_{02} = 1$, $ x_{10} = 1$, $ x_{12} = -1$, $ x_{20} = -1$, $ x_{21} = 1$, $ y_{10} = 1$, $ y_{12} = -1$.

 {\it The rather curious conclusion is that the monodromy space, for the case that $z=0$ is regular, consists of a single point.}

The formal matrix differential operator is
\[
z\frac{{\rm d}}{{\rm d}z}+ \begin{pmatrix} -\frac{1}{3} &tz &z^2 \\ z&0 & tz \\ t& z& \frac{1}{3} \end{pmatrix}.\]
The {\it guess} that $\mathcal{M}$ is represented by the family of operators of the form
\[\frac{{\rm d}}{{\rm d}z}+ {\begin{pmatrix} 0 &\frac{3t}{2} &z \\ 1&0 & \frac{3t}{2} \\ 0& 1& 0 \end{pmatrix}}
\]
is {\it confirmed} by a Lax pair computation. The corresponding scalar differential equation is $y^{(3)}-3ty^{(1)}-zy=0$.

 {\it One concludes that the Stokes
matrices, which are nontrivial, in this family do not depend on $t$.}

 \section[z plus sqrt(tz), z-sqrt(tz), -2z and regular singular z=0]{$\boldsymbol{z+t^{1/2}z^{1/2}, z-t^{1/2}z^{1/2},-2z}$ and regular singular $\boldsymbol{z=0}$}\label{four}
 The above formulas are the eigenvalues $q_0$, $q_1$, $q_2$ at $z=\infty$.
For $t\in \mathbb{R}$, $t>0$, the singular directions in $[0,1)$ are $1/2$ for $q_0-q_2$, $q_1-q_2$ and $0$ for $q_1-q_0$, $q_2-q_0$, $q_2-q_1$. The monodromy identity is
\[\operatorname{mon}= \begin{pmatrix} 0 &-1 & 0\\ 1&0 &0 \\ 0 &0 &1 \end{pmatrix}
\begin{pmatrix} 1 &0 &0 \\ 0&1 &0 \\ x_{02}& x_{12}&1 \end{pmatrix}
\begin{pmatrix} 1 &x_{10} &x_{20} \\ 0&1 &x_{21} \\ 0&0 &1 \end{pmatrix}.\]
This presentation is unique up to scaling the third basis vector.
One has $\dim \mathcal{R}=4$, $\dim \mathcal{P}=2$ and the fibers are affine cubic surfaces
with an equation of the form $x_1x_2x_3+x_1^2+*x_1+*x_2+*x_3+*=0$.
The data do not match with $z=0$ regular, since $\operatorname{mon}={\bf 1}$ has no solution.

In order to simplify the computation, the eigenvalues at $z=\infty$ are replaced by $z^{1/2}$, $-z^{1/2}$, $tz$, obtained by shifting and scaling
the given $z\pm t^{1/2}z^{1/2}$, $-2z$. This has no effect on monodromy and Lax pairs.
 The method of \cite[Section~12.5]{vdP-Si}, explained in
Section~\ref{section0}, produces an explicit formula for the universal family with these data.
 It depends on 5~variables. By scaling the basis vectors, one of the variables is normalized to 1 and the resulting operator has the form
$z\frac{{\rm d}}{{\rm d}z}+ A$ with
\[ A= \begin{pmatrix} -2a_1&z& a_4t-a_3\\ 1&2a_1+a_3t & a_3t \\ -(a_1+2a_8)t-1& tz & tz-a_3t \end{pmatrix}.\]
The further computations are described as follows.
The Lax pair formalism produces a set of differential equations $\frac{{\rm d}a_i}{{\rm d}t}=R_i$ for $i=1,3,4,8$ where the $R_i$ are rational
functions in $a_1$, $a_3$, $a_4$, $a_8$, $t$. The eigenvalues of the residue matrix of $A$ at $z=0$ are independent of $t$. One of the
eigenvalues is $2a_1+a_3t$. Adding the equation $\frac{{\rm d}(2a_1+a_3t) }{{\rm d}t}=0$ to the above system of equations eliminates $a_4$
and produces: $a_1=c_1$, $a_3=c_2/t$ with constants $c_1$, $c_2$ and a Riccati equation for~$a_8$. In particular,
 {\it The Painlev\'e equation of this family is solvable by classical functions.}

We remark that the monodromy space for the
case $\pm z^{1/2}$, $tz$ produces again a 2-dimensional family of affine cubic surfaces of the type:
$x_1x_2x_3+x_1^2+*x_1+*x_2+*x_3+*=0$.

However, in this case, there seems to be no relation with $P_4$.

 \section[e1-th root of z, t times e2-th root of z, e1 at least e2 at least 2 and regular singular z=0]{ $\boldsymbol{z^{1/e_1}}$, $\boldsymbol{tz^{1/e_2}}$, $\boldsymbol{e_1\geq e_2\geq 2}$ and regular singular $\boldsymbol{z=0}$}\label{five}

 For $e_1>e_2\geq 2$, one has $\dim \mathcal{R}=3e_1+e_2-3$ and $\dim \mathcal{P}=e_1+e_2-1$. For the smallest case $e_1=3$, $e_2=2$,
 one has $\dim \mathcal{R}=8$, $\dim \mathcal{P}=4$, $\mathcal{R}\rightarrow \mathcal{P}$ is
(generically) surjective and the fibres have dimension 4. The formulas for the fibres are
complicated. This makes the computation of~$\mathcal{M}$ and the Lax pair nearly impossible.

For $e_1=e_2=m\geq 2$, $ n=2m$, one has $\dim \mathcal{R}=\frac{n(n-1)}{m}-1=2n-3$, $\dim \mathcal{P}=n-1$.
For the smallest case $m=2$, one has $\dim \mathcal{R}-\dim \mathcal{P}=2$. The computation in Section~\ref{5.1} below produces a
second-order Painlev\'e equation which is probably a pull back of the classical $P_4$ equation.

\subsection[The case sqrt(z), t*sqrt(z). The monodromy space R]{The case $\boldsymbol{z^{1/2}}$, $\boldsymbol{tz^{1/2}}$. The monodromy space $\boldsymbol{\mathcal{R}}$}\label{5.1}

 The eigenvectors are $q_1=z^{1/2}$, $q_2=-z^{1/2}$, $q_3=tz^{1/2}$, $q_4=-tz^{1/2}$ with a basis $f_1$, $f_2$, $f_3$, $f_4$
 of eigenvectors is such that $\gamma$ permutes the two pairs $\{f_1,f_2\}$ and $\{f_3,f_4\}$. There are 6 variables present in the Stokes matrices. Now $f_1$ and $f_3$ can be scaled independently and so $\dim \mathcal{R}=6-1=5$. Further
 $\mathcal{P}\cong \mathbb{C}^3$ is the space of the characteristic polynomials $T^4-p_3T^3+p_2T^2-p_1T+1$ of
 the elements of ${\rm SL}_4$.

For the case $t={\rm i}$, the singular directions in $[0,1)$ are $0$ for $q_2-q_1$, $\frac{1}{4}$ for $q_2-q_4$ and for $q_3-q_1$,
$\frac{1}{2}$ for $q_3-q_4$, $\frac{3}{4}$ for $q_1-q_4$ and for $q_3-q_2$.

 The topological monodromy $\operatorname{mon}$, which has the above characteristic polynomial, is equal to the
product
\begin{gather*}
\begin{pmatrix} &1 & & \\ 1& & & \\ & & &1 \\ & & 1& \end{pmatrix}
\begin{pmatrix} 1& & &x_{14} \\ & 1&x_{23} & \\ & &1 & \\ & & &1 \end{pmatrix}
\begin{pmatrix} 1& & & \\ & 1& & \\ & & 1& \\ & &x_{43} &1 \end{pmatrix}
 \begin{pmatrix} 1& & x_{13}& \\ & 1& & \\ & & 1& \\ & x_{42}& & 1 \end{pmatrix}
\begin{pmatrix} 1&x_{12} & & \\ & 1& & \\ & & 1& \\ & & & 1 \end{pmatrix}.
\end{gather*}
 The fibres of $\mathcal{R}\rightarrow \mathcal{P}$ have the following data.
Assume that $x_{42}\neq 0$ (we note that $x_{42}=0$ implies reducibility). After eliminating $x_{12}$, $x_{23}$, there remains one
cubic equation in the variables $x_{14}$, $x_{42}$, $x_{43}$, $x_{13}$, namely
$x_{14}x_{42}x_{43} - p_3x_{43} + x_{13}x_{42} + x_{43}^2 + p_2 + 2=0$.
After normalizing $x_{13}$ to $1$ (we note that $x_{13}=0$ produces a similar cubic equation), the cubic equation is almost identical
to the one for $P_4$ (see \cite{vdP-Sa}) which is
$x_1x_2x_3+x_1^2-\big(s_2^2 +s_1s_3\big)x_1-s_2^2x_2-s_2^2x_3+s_2^2+s_1s_2^3=0$.

{\it Observation}. The monodromy identity depends strongly on $t$. For instance, if $t\in \mathbb{R}_{>0}$, $t\neq 1$, then there is only one singular direction in $[0,1)$. The fibres of $\mathcal{R}\rightarrow \mathcal{P}$ are again rational surfaces.
 The dependence of these surfaces on $t$ is somewhat mysterious.

{\it The space of connections $\mathcal{M}$}.
 A differential module $M$ over $\mathbb{C}(z)$ in this moduli space can be considered as a differential module $N$ over $\mathbb{C}\bigl(z^{1/2}\bigr)$
 with an automorphism $\sigma$ satisfying $\sigma \circ z^{1/2}=\smash{-z^{1/2}}\circ \sigma$ and $\sigma ^2=1$. Now $N$ can be given a basis $e_1$, $e_2$, $e_3$, $e_4$ such that $\sigma$
 permutes the two pairs $\{e_1,e_2\}$ and $\{e_3,e_4\}$. The corresponding module $M$ over $\mathbb{C}(z)$ has the basis
 \[ B_1=e_1+e_2,\qquad B_2:=z^{1/2}(e_1-e_2),\qquad B_3=e_3+e_4,\qquad B_4=z^{1/2}(e_3-e_4).\]
 Let $D$ denote the operator of the form \smash{$z\frac{{\rm d}}{{\rm d}z}+(\textit{a matrix})$}
 acting upon $N$. This operator commutes with $\sigma$ and is determined by $De_1$, $De_3$. The formal part $\hat{D}$ of $D$ is given by
 \smash{$\hat{D}e_1=z^{1/2}e_1$}, \smash{$\hat{D}e_3=tz^{1/2}e_3$}. Using \cite[Section~12]{vdP-Si}, one concludes that $D$ is given by the formulas
 \[ De_1=z^{1/2}e_1+a_1e_2+a_2e_3+a_3e_4,\qquad De_3=tz^{1/2}e_3+a_4e_1+a_5e_2+a_6e_4 \]
 with constants $a_1,\dots ,a_6$. For the generic case, one can normalize to $a_5=1$. A computation of~$D$ on the basis
 $B_1,\dots , B_4$ produces the operator (normalized to trace equal to zero)
 $ z\frac{{\rm d}}{{\rm d}z}+A$ with
 \[A=\begin{pmatrix} a_1-1/4 & z& a_4+a_5&0 \\ 1& -a_1+1/4&0 & a_4-a_5\\ a_2+a_3& 0&a_6-1/4 & tz\\ 0&a_2-a_3 & t & -a_6+1/4 \end{pmatrix}.\]

 As mentioned above, one may normalize to $a_5=1$. Further the coefficients of the characteristic polynomial of residue matrix at $z=0$ are
 the parameters. We conclude that for fixed parameters, the above family of operators has dimension 2 (not counting the variable $t$). This is in
 agreement with the computation of the fibers of $\mathcal{R}\rightarrow \mathcal{P}$.

 The operator $z\frac{{\rm d}}{{\rm d}z}+A$ is extended to a Lax pair by an operator of the form $\frac{{\rm d}}{{\rm d}t}+B$ with $B=B_0+B_1z$ for matrices $B_0$, $B_1$ depending on $t$ only. The property $\frac{{\rm d}}{{\rm d}t}(A)=z\frac{{\rm d}}{{\rm d}z}(B)+[A,B]$ yields a vector field, represented by differential equations
 $\frac{{\rm d}}{{\rm d}t}a_j=R_j$, $ j=1,2,3,4,6$ with $R_j$ rational expressions in $a_1,\dots, a_4, a_6,t$. The characteristic polynomial of the residue matrix at $z=0$ is written as $T^4+P_2T^2+P_1T+P_0$. We have used the formulas for $P_1$, $P_2$ and another invariant $P_0=a_1+a_6$ to
 eliminate (stepwise) the functions $a_3$, $a_6$, $a_2$. For the remaining $a_1$, $a_4$ one obtains the equations
\begin{gather*}
\frac{{\rm d}a_1}{{\rm d}t}= \frac{\bigl(2P_1a_4^2 + \bigl(-4P_0^2 + (8a_1 + 2)P_0 - 8a_1^2 - 4P_2\bigr)a_4 + 2P_1\bigr)}{\bigl(a_4^2 - 1\bigr)\bigl(t^2 - 1\bigr)}, \\
 \frac{{\rm d}a_4}{{\rm d}t}= \frac{ \bigl(\bigl(-2a_4^2 + 2\bigr)a_1t^2 + 2\bigl(a_4^2 + 1\bigr)(P_0 - 2a_1)t + 2(a_4 - 1)(a_4 + 1)(-a_1 + P_0))} {\bigl(t^3 - t\bigr)}.
 \end{gather*}
The second equation can be used to write $a_1$ as an expression in $a_4$ and $\frac{{\rm d}a_4}{{\rm d}t}$. Substitution in the first equation
yields an explicit (and rather long) second-order differential equation for $a_4$. The poles with respect to $t$ are $0$, $1$, $-1$, $\infty$.
The cubic form of the fibres of $\mathcal{R}\rightarrow \mathcal{P}$ suggested a relation with~$P_4$.

 Identifying Painlev\'e equations is in general an almost impossible task; the present manuscript focuses on constructing Painlev\'e-type equations rather than identifying them. However, Dzha\-may~\cite{Dz} succeeded in using the
 geometry of Okamoto--Painlev\'e spaces and the algorithms of~\cite{D-F-L-S}, to identify the above equation with
 the standard Okamoto form $P_{\rm VI}(q(t), p(t), t; \kappa_0, \kappa_1, \kappa_8, \allowbreak\theta)$ of the sixth Painlev\'e
 equation, where
\begin{gather*}
a_4(s) = (s\cdot q(t) - 1)/(s \cdot q(t) + 1), \qquad a_1(s) = 1/2 (-2 p(t) + P_0 - \kappa_1 ),\qquad
t = 1/s^2,\\
 P_0 = 1/2 - \kappa_1,\qquad P_1 = \bigl(\kappa_8^2 - \kappa_0^2\bigr)/4,\qquad P_2 = \bigl(2P_0 - 2P_0^2 - \kappa_0^2 - \kappa_8^2\bigr)/4.
 \end{gather*}

\section[(z-cubed plus tz) with multiplicity m1, -m1*(z-cubed plus tz)/m2 with multiplicity m2 and regular z=0]{$\boldsymbol{\bigl(z^3+tz\bigr)_{m_1}}$, $\boldsymbol{\bigl(\frac{-m_1}{m_2}\bigl(z^3+tz\bigr)\bigr)_{m_2}}$ and regular $\boldsymbol{z=0}$}\label{six}

 {\it The case $m_1=m_2=1$}. This is in fact the standard family for $P_2$. One has $q_1=z^3+tz$, $q_2=-\bigl(z^3+tz\bigr)$; the singular directions for $q_1-q_2$ are $\frac{1,3,5}{6}$; the singular directions for
 $q_2-q_1$ are $\frac{0,2,4}{6}$. The monodromy identity reads
 \[\operatorname{mon}=\begin{pmatrix}g& 0\\ 0 &1/g \end{pmatrix} \begin{pmatrix}1& x_5\\0& 1\end{pmatrix} \begin{pmatrix}1& 0\\x_4& 1\end{pmatrix} \begin{pmatrix}1& x_3\\0& 1\end{pmatrix} \begin{pmatrix}1& 0\\x_2& 1 \end{pmatrix}\begin{pmatrix}1& x_1\\0& 1 \end{pmatrix} \begin{pmatrix}1& 0\\x_0& 1\end{pmatrix}.\]
 The equation $\operatorname{mon}=\bigl(\begin{smallmatrix}1& 0\\0& 1\end{smallmatrix}\bigr)$ and division by $\mathbb{G}_m$ (made explicit by normalizing $x_5=1$) produce an explicit
 $\mathcal{R}$ of dimension 3. Furthermore, $\dim \mathcal{P}=1$ (parameter $g$). See \cite[Section~3.9]{vdP-Sa} for more details.

{\it Case $m_1=2$, $m_2=1$, $z^3+tz$, $z^3+tz$, $-2z^3-2tz$ and $z=0$ is regular.}

{\it Description of $\mathcal{R}$}.
The Stokes matrices are described by 12 variables, there are 2 variables describing the formal monodromy. The topological monodromy is
supposed to be the identity. This produces 8 equations.

Actual computation, using the monodromy identity produces a space of dimension 6. Dividing by the action of
$\mathbb{G}_m^2$, due to scaling the basis vectors, produces $\dim \mathcal{R}=4$. Moreover, $\dim \mathcal{P}=2$ and the fibers of
$\mathcal{R}\rightarrow \mathcal{P}$ have dimension 2. One expects a relation with a classical Painlev\'e equation with at most two singularities.

We follow the discussion of the cases (i) and (ii) on pages~\pageref{(i)}--\pageref{(ii)} for the construction of the matrix differential operator $z\frac{{\rm d}}{{\rm d}z}+A$.
The formal operator
\[
{\rm ST}:=z\frac{{\rm d}}{{\rm d}z}+\operatorname{diag}\bigl(z^3+tz+a_1,z^3+tz+a_2,-2z^3-2tz-a_1-a_2\bigr)
\]
 is conjugated with
the matrix ${\bf 1}+M$, where
\[ M=\begin{pmatrix} 0&0 & b_1/z + b_2/z^2 + b_3/z^3\\
0 &0 & b_4/z + b_5/z^2 + b_6/z^3 \\
b_7/z + b_8/z^2 + b_9/z^3 &b_{10}/z + b_{11}/z^2 + b_{12}/z^3 &0 \end{pmatrix}.\]
This produces an operator
\[
z\frac{{\rm d}}{{\rm d}z}+\tilde{A} = \operatorname{Prin}\left( z\frac{{\rm d}}{{\rm d}z}+( {\bf 1}+M){\rm ST}({\bf 1}+M)^{-1}\right).
\]
Then $A$ is obtained from $\tilde{A}$
by adding the eight equations given by $\tilde{A}(0)=0$. After scaling the basis vectors one has~${b_2=1}$, $b_{11}=1$ and the matrix
$A$ depends only on the variables $b_1$, $b_5$, $b_7$,~$b_8$. The substitution~${b_5=B_5b_8}$ removes some denominators.
The characteristic polynomial of the formal monodromy at $z=\infty$ is written as $T^3+pT+q$.
The Lax pair equation
together with the equations $\frac{{\rm d}p}{{\rm d}t}=0$, $\frac{{\rm d}q}{{\rm d}t}=0$ produce a differential system
\begin{gather*} \frac{{\rm d}B_5}{{\rm d}t}=0, \frac{{\rm d}b_1}{{\rm d}t}= -6B_5b_1^3b_7 - 6b_1^3b_7 - 3b_1^2t - 3, \\
\frac{{\rm d}b_7}{{\rm d}t}=\bigl(27b_7^2(B_5 + 1)^2b_1^2 + 18tb_7(B_5 + 1)b_1 + 3t^2 - p\bigr)/(3B_5 + 3) .\end{gather*}
This is a Hamiltonian system
$\frac{{\rm d}b_1}{{\rm d}t}=\frac{\partial H}{\partial b_7}$, $\frac{{\rm d}b_7}{{\rm d}t}=-\frac{\partial H}{\partial b_1}$,
depending on parameters $B_5$, $p$, with
\[
H=-3(B_5 +1)b_7^2b_1^3 - 3tb_7b_1^2 + \frac{-3t^2 + p}{3B_5 + 3} b_1 - 3b_7.
\]

There is no evident relation between this Hamiltonian and Okamoto's list of polynomial
Hamiltonians in \cite{O1}. Possibly the method described in \cite{D-F-L-S} can be applied here.

 \section[z, tz, (-1-t)z, J.~Harnad's case ]{$\boldsymbol{z}$, $\boldsymbol{tz}$, $\boldsymbol{(-1-t)z}$, Harnad's case}\label{seven}
The moduli space $\mathcal{M}$ for these data is a family with irregular singularities and its isomonodromy produces
the Painlev\'e VI equation. Classically, $P_6$ is derived from isomonodromy with four regular singularities.
The construction and formulas for this new family are introduced by Harnad~\cite{H}. A more detailed investigation
 is given by Mazzocco \cite{Maz}. A formula for the Stokes data of this new family in terms of
 invariants for the classical family is computed in \cite{D-G}.

 In this section, we compute the Stokes data, the monodromy space $\mathcal{R}$, a matrix differential operator
 representing $\mathcal{M}$, its identification with the data in \cite{D-G,H, Maz} and the Lax pair.
 A direct identification of the resulting Painlev\'e type equations with $P_6$ seems difficult to find. However
 this identification is explicitly present in \cite{H, Maz}.

 \subsection[Computation of Stokes data and monodromy space R]{Computation of Stokes data and monodromy space $\boldsymbol{\mathcal{R}}$}

The formal solution space $V$ at $z=\infty$ has a basis $e_0$, $e_1$, $e_2$ such that the formal differential operator has the form
\[ z\frac{{\rm d}}{{\rm d}z}+\begin{pmatrix} z+a_0 & & \\ &tz +a_1& \\ & &(-1-t)z-a_0-a_1 \end{pmatrix}.\]
The formal monodromy and the Stokes matrices are given with respect to this natural basis.
The basis is unique up to multiplying each $e_j$ by a scalar. Since $z=0$ is regular singular, one has $\dim \mathcal{R}=6$.
 The formal monodromy at $z=\infty$ is the diagonal matrix $\operatorname{diag}(g_1,g_2,g_3)$
 with~${g_1g_2g_3=1}$ and $g_1={\rm e}^{2\pi {\rm i} a_0}$, $g_2={\rm e}^{2\pi {\rm i} a_1}$.

The singular directions depend on $t$. For $t$ close to ${\rm i}$, the singular direction $d_{kl} \in [0,1)$ for $ q_k-q_l$ are approximated
by 0.93, 0.83, 0.62, 0.43, 0.33, 0.12 for $d_{20}$, $d_{21}$, $ d_{01}$, $d_{02}$, $d_{12}$, $d_{10}$.
This determines the order of the six Stokes matrices in the monodromy identity, which states that the topological monodromy $\operatorname{mon}$ at
$z=0$ is, up to conjugation, equal to the product
\begin{gather*}
\begin{pmatrix} g_1& & \\ &g_2 & \\ & &\frac{1}{g_1g_2} \end{pmatrix}
\begin{pmatrix} 1& & \\ &1 & \\ x_{20}& &1 \end{pmatrix}
\begin{pmatrix} 1 & & \\ &1 & \\ &x_{21} & 1 \end{pmatrix}
\begin{pmatrix} 1& x_{01}& \\ &1 & \\ & &1 \end{pmatrix} \\
\qquad\times\begin{pmatrix} 1 & & x_{02}\\ &1 & \\ & & 1 \end{pmatrix}
\begin{pmatrix} 1& & \\ &1 & x_{12}\\ & & 1 \end{pmatrix}
\begin{pmatrix} 1& & \\ x_{10}&1 & \\ & & 1 \end{pmatrix}. \end{gather*}

 If $z=0$ is regular (i.e., $\operatorname{mon}={\bf 1}$), then the Stokes matrices and the formal monodromy are equal to the identity. This case is uninteresting. We suppose that $z=0$ is any regular singularity.

Since the basis $e_0$, $e_1$, $e_2$ is unique up to multiplying each $e_j$ by a constant, the monodromy space
$\mathcal{R}$ is the quotient of the space of tuples
$\{g_1,g_2,x_{20}, x_{21},x_{01}, x_{02}, x_{12},x_{10}\}$ by the action of~$\mathbb{G}_m^2$. A dense affine
subspace of $\mathcal{R}$ is obtained by normalization two of the $x_{**}$ to $1$, for instance,~${x_{20}=x_{12}=1}$.
It can be shown that $\mathcal{R}\rightarrow \mathcal{P}$ is surjective. The space $\mathcal{P}$ is given by the
tuples $(g_1,g_2, c_1,c_0)$ where $c_0$, $c_1$ are the nontrivial coefficients of the characteristic polynomial of $\operatorname{mon}$.
Now $\dim \mathcal{P}=4$ and a computation shows that the fibers of $\mathcal{R}\rightarrow \mathcal{P}$ are
affine cubic surfaces with equation $xyz+x^2+y^2+z^2+p_1x+p_2y+p_3z+p_4=0$ with $p_1,\dots ,p_4$ expressions in the
parameters $g_1$, $g_2$, $c_1$, $c_0$.

 This is the {\it expected} structure of the monodromy space if one admits the equivalence with the classical
 isomonodromy for $P_6$ (see the list in \cite[Section~2.2]{vdP-Sa}).

\subsection{Constructing the connection and the Lax pair}

 A Zariski open, dense subspace of the moduli space $\mathcal{M}$ is obtained from \cite[Theorem 12.4]{vdP-Si} (see also Section~\ref{section0}) by considering the family of differential operators of the form
 \[ z\frac{{\rm d}}{{\rm d}z}+\begin{pmatrix}a_0 &m_1 &m_2 \\m_3 &a_1 &m_4 \\ m_5 &m_6 &-a_0-a_1 \end{pmatrix}
 +z\begin{pmatrix}1 &0 &0 \\ 0& t& 0\\ 0& 0&-1-t \end{pmatrix}.\]

Changing the eigenvalues $1$, $t$, $-1-t$ of the irregular part of the operator into 0, 1, $t$ has no effect on the
monodromy and the Lax pair. One obtains in this way the \cite[formula (3.62)]{H} proposed by Harnad and the formula on \cite[p.\ 3]{Maz}.

The Zariski open subspace of $\mathcal{M}$ is obtained by taking equivalence classes of the above family.
Indeed, each of the basis vectors for this presentation can be multiplied by nonzero elements and therefore the family has to
 be divided by this action of $\mathbb{G}_m^2$. A Zariski open part of the quotient space is obtained by assuming
 $m_3m_4\neq 0$ and normalizing $m_3=m_4=1$. In this way, one obtains the following explicit description of an open part of
 $\mathcal{M}$ in terms of
\[ z\frac{{\rm d}}{{\rm d}z}+\begin{pmatrix} z+a_0&v_1 &v_2 \\ 1 & tz+a_1 &1 \\ v_3& v_4 & (-1-t)z-a_0-a_1\end{pmatrix}.\]
This operator, with $v_1,\dots ,v_4$ as functions of $t$ and $a_0$, $a_1$ constants,
 is completed to a Lax pair with the operator $\frac{{\rm d}}{{\rm d}t}+B_0(t)+zB_1(t)$. The assumption that the two operators
 commute leads to a set of differential equations for $v_1,\dots ,v_4$, namely
 \begin{gather*}
 v_1' = \frac{-3v_3v_1+3v_2v_4}{2t^2+5t+2}, \qquad
 v_2' = \frac{(6t+3)v_2^2-3v_3(t-1)v_2-3v_1(t+2)+(-9a_0t+9a_1)v_2}{2t^3+3t^2-3t-2}, \\
 v_3' = \frac{(3t-3)v_3^2+(-6t-3)v_2v_3+(9a_0t-9a_1)v_3+3v_4(t+2)}{2t^3+3t^2-3t-2},\qquad
 v_4' = \frac{3v_3 v_1-3v_2v_4}{t^2+t-2}. \end{gather*}
 In a monodromic family the topological monodromy is constant and then also the characteristic polynomial of the residue matrix
 is constant. This means that there are constants $\delta_0$, $\delta_1$, explicitly
 \begin{gather*} -a_0^2-a_0a_1-a_1^2-v_2v_3-v_1-v_4=\delta_1,\\
 a_0^2a_1+a_0a_1^2+a_1v_2v_3-a_0v_1+a_0v_4-a_1v_1-v_1v_3-v_2v_4=\delta_0. \end{gather*}
The algebra of functions on the parameter space $\mathcal{P}^+$ for the connection
is generated by $a_0$, $a_1$, $\delta_0$, $\delta_1$. They correspond to the
 4 parameters for the moduli space $\mathcal{R}$ of the analytic data.

We know by \cite{H,Maz}, that reduction of the above equations and parameters to the same for~$P_6$ is possible.
However, we have no explicit computation.

\begin{Remark} There are explicit formulas for reducible loci and the corresponding Riccati equations. These turn out to be
hypergeometric differential equations. This is also expected if one admits that the Painlev\'e type system is
equivalent to $P_6$.
\end{Remark}

\begin{Example}
 $v_3=v_4=0$ and $v_1'=0$, \[
 v_2'=\frac{3(2t+1)v_2^2-3(t+2)v_1+9(-a_0t+a_1)v_2}{(t-1)(2t+1)(t+2)}.
 \]
\end{Example}

\subsection[The hierarchy z-sub(m1), (tz)-sub(m2), ((-m1-tm2)z/m3(-sub(m3)]{The hierarchy $\boldsymbol{(z)_{m_1}}$, $\boldsymbol{(tz)_{m_2} (\frac{-m_1-tm_2}{m_3}z)_{m_3}}$}

The $m_1,m_2,m_3\geq 1$ stand for multiplicities. A computation shows that ``$z=0$ is regular'' leads to trivial Stokes matrices and
formal monodromy. Consider the case $m_1=2$, $m_2=m_3=1$. Then $\dim \mathcal{R} =10$ and $\dim \mathcal{P}=6$.
According to \cite[Exercise~12.5, p.~300]{vdP-Si}, the universal family is represented by the operator
\[z\frac{{\rm d}}{{\rm d}z} +
\begin{pmatrix} z+a_1 &0 & x_1&x_2 \\ 0 &z+a_2 & x_3&x_4 \\ x_5& x_6& tz+a_3&x_7 \\ x_8&x_9 &x_{10} & (-2-t)z-a_1-a_2-a_3\end{pmatrix} .\]
 Consider a normalization, say, $x_8=x_9=x_{10}=1$, obtained by restricting to the open subspace defined by $x_8x_9x_{10}\neq 0$
and multiplying the base vectors by scalars. The Lax pair consists of the above operator and $\frac{{\rm d}}{{\rm d}t}+B_0+zB_1$.
An easy computation produces a system of differential equations (or vector field)
 \smash{$\frac{{\rm d}x_j}{{\rm d}t}\in \mathbb{C}(a_1,\dots,a_4,x_1,\dots ,x_7)$}, $ j=1,\dots, 7$. Using three coefficients of the characteristic polynomial of the residue matrix at $z=0$, one can eliminate $x_2$, $x_4$, $x_7$ in a~rational way.
The resulting vector field of rank 4 is a Painlev\'e type system. It is unfortunately too complicated for presentation here.
This system contains, of course, many closed subsystems corresponding to $z$, $tz$, $(-1-t)z$, that produce
equations related to $P_6$.

\section[e-th root of z, tz, -tz for e at least 2 and regular singular z=0]{$\boldsymbol{z^{1/e}}$, $\boldsymbol{tz}$, $\boldsymbol{-tz}$, $\boldsymbol{e>1}$ and regular singular $\boldsymbol{z=0}$}\label{sec9}

Write $q_1,\dots , q_e$ for the conjugates of \smash{$z^{1/e}$}; write $r$ and $s$ for $tz$ and $-tz$. The value of the integer~$N$ (i.e., the dimension of the Stokes data) is equal to $ e(e-1)\frac{1}{e}$ (for the $q_k-q_l$) plus $e+e+e+e$ (for $q_k-r$, $r-q_k$, $q_k-s$, $s-q_k$) plus $2$ (for $r-s$, $s-r$) and sums up to~${5e+1}$. The formal monodromy~$\gamma$ depends on 1 parameter. Normalization by the action of $\mathbb{G}_m^2$
results in $\dim \mathcal{R}=5e$. The parameters are the $e+1$ coefficients of the characteristic polynomial of the monodromy at~${z=0}$ and the eigenvalues of $\gamma$. This leads to $\dim \mathcal{P}=e+2$ and $\dim \mathcal{R}-\dim \mathcal{P}=4e-2$.

{\it The case $e=2$}. For $t=1$, the singular directions in $[0,1)$ are
$d=1/2$ for $r-q_1$, $r-q_2$, $q_1-s$, $q_2-s$, $r-s$ and $q_2-q_1$, $q_1-r$, $q_2-r$, $s-q_1$, $s-q_2$, $s-r$ for $d=0$. This leads to a matrix formula for the topological monodromy
\[\operatorname{mon}_0=\begin{pmatrix} & -1& & \\ 1& & & \\ & & g& \\ & & & 1/g\end{pmatrix}
\begin{pmatrix} 1 & & x_1& \\ & 1& x_2& \\ & & 1& \\ x_3&x_4 & x_5& 1\end{pmatrix}
\begin{pmatrix} 1&x_6 & & x_7\\ & 1& &x_8 \\ x_9&x_{10} & 1&x_{11} \\ & & & 1\end{pmatrix}. \]
One may normalize to $x_{10}=x_{11}=1$. The parameters are $g$ and three coefficients of the characteristic polynomial of $\operatorname{mon}_0$.
A Maple computation verifies that the fibres of $\mathcal{R}\rightarrow \mathcal{P}$ are
 birational to $\mathbb{A}^6$.

The form of the differential operator $z\frac{{\rm d}}{{\rm d}z}+A$ is found, using the method explained in
Section~\ref{section0}, namely
\[ A= \begin{pmatrix} a&z & *&* \\ 1&-a & *&* \\ *&* &tz+b &* \\ *&* &* &-tz-b \end{pmatrix},\]
 where $a$, $b$ and $*$ are variables. Since two of the constants can be normalized to $0$, this is a family of dimension 10
 (over the field $\mathbb{C}(t)$).
 The invariants are three coefficients of the characteristic polynomial of the residue matrix at $z=0$ and $b$
 at $z=\infty$.

The Lax pair $\bigl\{ z\frac{{\rm d}}{{\rm d}z}+A,\frac{{\rm d}}{{\rm d}t}+B\bigr\}$ has the form $B=B_0+B_1z$, where $B_0$, $B_1$ are traceless matrices depending on $t$ only. The Lax pair calculations for the case $e=2$ produces a rational Painlev\'e vector field of dimension 6 which is too large to be presented here.

\section[z*z, -z*z-tz, tz and z=0 regular]{$\boldsymbol{z^2}$, $\boldsymbol{-z^2-tz}$, $\boldsymbol{tz}$ and $\boldsymbol{z=0}$ regular}\label{nine}

The assumptions: $z=0$ is regular and $z=\infty$ is unramified and has Katz invariant 2 produces the eigenvalues
\smash{$\bigl(z^2\bigr)_{m_1}$}, \smash{$\bigl(a_2z^2+a_1z\bigr)_{m_2}$}, \smash{$(b_1z)_{m_3}$} such that the $m_1$, $m_2$, $m_3$ satisfy
\[
m_1z^2+m_2\bigl(a_2z^2+a_1z\bigr)+m_3b_1z=0.
\]
 Then $\dim \mathcal{R}=(n-1)^2$, where $n=m_1+m_2+m_3\geq 3$.

 We present computations for the case
$m_1=m_2=m_3=1$. Then $\dim \mathcal{R}=4$ and $\dim \mathcal{P}=2$. The fibers of $\mathcal{R}\rightarrow \mathcal{P}$ are
affine cubic surfaces, which have, after an affine linear change of the variables, the equation $xyz+x+y+1=0$.
We note that the monodromy space of standard family which produces the first Painlev\'e equation $P_1$ is the affine cubic surface with the same equation (see \cite[Section~3.10]{vdP-Sa}). Despite this similarity, we did not find a relation between the case under consideration and $P_1$.

We {\it propose} a normalized differential operator
\[ \frac{{\rm d}}{{\rm d}z}+ \begin{pmatrix} z &a_1 &1 \\ 1&-z-t &a_2 \\ a_3& a_4& t \end{pmatrix}.\]
The reasoning for this proposal is the following. We observe that
\[z\frac{{\rm d}}{{\rm d}z}+ \begin{pmatrix} z^2+c_1 &0 &0 \\ 0&-z^2-tz+c_2 &0 \\ 0& 0& tz-c_1-c_2 \end{pmatrix}\] has at $z=\infty$ the universal deformation
\[z\frac{{\rm d}}{{\rm d}z}+ \begin{pmatrix} z^2+c_1 &* &* \\ *&-z^2-tz+c_2 &* \\ *& *& tz-c_1-c_2 \end{pmatrix}.\]
Here the $*$'s are arbitrary polynomials in $z$ of degree $\leq 1$. The assumption that $z=0$ is regular implies that
$c_1=c_2=0$ and the $*$ are elements of $\mathbb{C}z$. Finally, in the general case one can, by a change of the basis,
 arrive at two entries being $z$. Dividing by $z$ produces the above proposal.

For the Lax pair situation, $a_1$, $a_2$, $a_3$, $a_4$ are functions of $t$ and the above operator is supposed to commute with $\frac{{\rm d}}{{\rm d}t}+B(z,t)$, where
$B(z,t)$ has degree 1 in the variable $z$. The resulting differential equations are
\begin{gather*}
a_1' = -3a_1a_2a_3 + 3a_4, a_2' = -\frac{3}{2}a_1a_2^2 + 3a_2^2a_3 - \frac{9}{2}a_2t - 3/2,\\
 a_3' = \frac{3}{2}a_3a_1a_2 - \frac{3}{2}a_4, a_4' = \frac{3}{2}a_4a_1a_2 - 3a_4a_2a_3 + \frac{3}{2}a_1a_3 + \frac{9}{2}a_4t.\end{gather*}
The two parameters describing the parameter space
$\mathcal{P}$ are $p_1=a_1+2a_3$ and $p_2:=a_1+a_3+a_2a_4$ (thus $p_1'=0$, $p_2'=0$).
Elimination of $a_1$, $a_3$ leads to the system of differential equations
\begin{gather*}
\begin{split}
&a_2'=-\frac{3}{2}+6a_2^3a_4+\frac{9p_1-12p_2}{2}a_2^2-\frac{9a_2t}{2},\\
& a_4'=\frac{3(-p_1+2p_2)(-p_2+p_1)}{2} +\frac{9t}{2}a_4+(-9p_1+12p_2)a_2a_4-9a_2^2a_4^2.
\end{split}
\end{gather*}
 The first equation can be used to write $a_4$ as rational expression in $a_2$, $a_2'$. Substitution in the second equation yields
 an explicit second-order equation. The Hamiltonian $H$ is equal to
\[-3a_4^2a_2^3-\frac{9p_1-12p_2}{2}a_4a_2^2-\frac{3p_1^2-9p_1p_2+6p_2^2}{2}a_2+\frac{3}{2}a_4+\frac{9a_2a_4}{2}t,\]
 where $a_4'=\frac{\partial H}{\partial a_2}$, $a_2'=-\frac{\partial H}{\partial a_4}$ and $p_1$, $p_2$ are parameters (constants).

 We did not find a relation with a classical Painlev\'e equation, but
 Dzhamay \cite[Section~2]{Dz}, using the geometry of Okamoto--Painlev\'e spaces
 and the algorithms of \cite{D-F-L-S}, computed that the case at hand corresponds to a $P_4$ equation. More precisely, the above equation is equivalent to the standard Okamoto form $P_{\rm IV} (q(t), p(t), t; \kappa_0, \theta_8 ) $ of the fourth Painlev\'e equation,
 via
$a_2(s) = -{\rm i} / q(t)$,
$a_4(s) = - {\rm i} q(t) (q(t) p(t) - \theta_8)$,
$s = 2 {\rm i} t /3$, and $p_1 = - 2 \kappa_0$, $p_2 = \theta_8 - 2 \kappa_0$.

\section[1/z, -1/z at z=0 and tz, -tz at z=infty]{$\boldsymbol{1/z,-1/z}$ at $\boldsymbol{z=0}$ and $\boldsymbol{tz}$, $\boldsymbol{-tz}$ at $\boldsymbol{z=\infty }$}\label{sec11}

The data above is the $m=1$ case of the hierarchy $\mathcal{M}_m$ defined for $m\geq 1$ by the eigenvalues
 $\bigl(z^{-1}\bigr)_m$, $-mz^{-1}$ at $z=0$ and $(tz)_m$, $-mtz$ at $z=\infty$. One shows that $\dim \mathcal{R}_m=4m$ and
 $\dim \mathcal{P}=2m$. The case $m=1$ is the family for $P_3(D_6)$.

 For $m=2$, a computation reveals that $\mathcal{R}\rightarrow \mathcal{P}$ is surjective with fibres of dimension~4. A~computation of $\mathcal{M}$, indicated in Section~\ref{section0},
and a choice of normalizations give rise to an operator $z\frac{{\rm d}}{{\rm d}z}+A$ representing an open affine subset of $\mathcal{M}$, with
 \[ A=z^{-1} \begin{pmatrix} 1&0 &0 \\ 0&1 & 0\\ 0&0 &-2 \end{pmatrix}+
\begin{pmatrix}c_1& 0&-3m_{1} \\ 0&c_2 & -3\\ 3& 3m_{3}&-c_1-c_2 \end{pmatrix}+
 tz P \begin{pmatrix} 1&0 &0 \\ 0&1 & 0\\ 0&0 &-2 \end{pmatrix}P^{-1} ,\]
where
\[
P=\left(\begin{matrix}1 &0 &x_1 \\ 0&1 &x_2 \\ x_3&x_4 & 1\end{matrix}\right).
\]
A straightforward Lax pair computation produces formulas for $\frac{{\rm d}x_i}{{\rm d}t}$, $i=1,\dots,4$, as rational functions in
$x_1$, $x_2$, $x_3$, $x_4$, $t$. These are however too large to be displayed here.
A computation verifies that in an isomonodromic family the $c_1$, $c_2$, $m_{3}$, $m_{1}$ are constant.

\section[1/sqrt(z) at z=0 and tz, -tz at z=infty]{$\boldsymbol{z^{-1/2}}$ at $\boldsymbol{z=0}$ and $\boldsymbol{tz}$, $\boldsymbol{-tz}$ at $\boldsymbol{z=\infty }$}\label{sec12}
The data above defines the usual family for $P_3(D_7)$. By attaching multiplicities, e.g., \smash{$\bigl( z^{-1/2}\bigr)_m$}, $(tz)_m$, $(-tz)_m$ and $m\geq 1$,
one obtains a hierarchy. For the case $m=2$, the space $\mathcal{R}$ is given by the equation
$L\circ \operatorname{mon}_0=\operatorname{mon}_{\infty}\circ L$ with $ L\colon V(0)\rightarrow V(\infty)$ a linear bijection
where~${\operatorname{mon}_0\colon V(0)\rightarrow V(0)}$, $ \operatorname{mon}_{\infty}\colon V(\infty )\rightarrow V(\infty) $ are
the topological monodromies. The link~$L$ is considered up to multiplication by $\mathbb{C}^*$
and the matrices of $\operatorname{mon}_0$ and $\operatorname{mon}_{\infty}$ have the form\looseness=-1
\begin{gather*}
\operatorname{mon}_0= \begin{pmatrix} & & * & \\ & & &* \\ 1& & & \\ & 1& & \end{pmatrix}
\begin{pmatrix} 1& & * & *\\ & 1& *& *\\ & & 1& \\ & & &1 \end{pmatrix}, \\
\operatorname{mon}_{\infty}=\begin{pmatrix} *& & & \\ & *& & \\ & & *& \\ & & & *\end{pmatrix}
\begin{pmatrix} 1& &* &* \\ & 1& * & *\\ & & 1& \\ & & &1 \end{pmatrix}
\begin{pmatrix} 1& & & \\ & 1& & \\ *&* & 1& \\ *& *& &1 \end{pmatrix}.\end{gather*}
After normalization, one obtains $\dim \mathcal{R}=12$ and $\dim \mathcal{P}=4$. Therefore, the relative dimension
of $\mathcal{M}$ over $\mathbb{C}(t)$ is also 12. The construction of a differential operator for
 $\mathcal{M}$ and the Lax pair computations
seem to be out of reach.

\section[1/(n-th root z) and t*(n-th root z), a hierarchy related to P3(D8)]{$\boldsymbol{z^{-1/n}}$ and $\boldsymbol{tz^{1/n}}$, a hierarchy related to $\boldsymbol{P_3(D_8)}$}\label{twelve}

The assumption that $z=0$ and $z=\infty$ are both irregular singular and totally ramified leads, after normalization, to the Galois orbit
of $z^{-1/n}$ at $z=0$ and the Galois orbit $tz^{1/n}$ at $z=\infty$. In the sequel, we replace $t$ by $t^{1/n}$. The moduli spaces will
be denoted by $\mathcal{M}_n$ and $\mathcal{R}_n$. The standard isomonodromic family for $P_3(D_8)$ is derived from
$\mathcal{M}_2$. First we study the structure of~$\mathcal{R}_n$.

\subsection[The structure of the monodromy space R\_n]{The structure of the monodromy space $\boldsymbol{\mathcal{R}_n}$}
We refer to Section~\ref{section0} and \cite[Sections~8 and 9]{vdP-Si} for notation and results.
 For a connection $M\in \mathcal{M}_n$, the solution space $V(\infty)$ at $z=\infty$ has the
 structure:
 \smash{$V(\infty)=\bigoplus_{j=0}^{n-1} \mathbb{C}e_j$} with $\mathbb{C}e_j= V(\infty)_{q_j}$,
 \smash{$q_j=\sigma^j\bigl(t^{1/n}z^{1/n}\bigr)=\zeta_n^jt^{1/n}z^{1/n}$},
 where $\zeta_n={\rm e}^{2\pi {\rm i}/n}$. The basis $\{e_j\}$ is chosen such that the formal monodromy $\gamma_{V(\infty)}$
 acts by $e_0\mapsto e_1 \mapsto \cdots \mapsto e_{n-1}\mapsto (-1)^{n-1}e_0$.

 By Lemma~\ref{lemma01}, the space of the Stokes data at $z=\infty$ can be identified with $\mathbb{C}^{n-1}$.
 The monodromy identity for the topological monodromy $\operatorname{mon}_\infty$ at $z=\infty$ has been studied
in detail in \cite[pp.~146--147]{CM-vdP}. The {\it surprising property} is:

 {\it Let the Stokes data be $(x_1,\dots ,x_{n-1})\in \mathbb{C}^{n-1}$. Then the characteristic polynomial
 of the topological monodromy $\operatorname{mon}_{\infty}$ is} \[ T^n+x_{n-1}T^{n-1}+\dots +x_1T+(-1)^{n}.\]
 Thus the map from the Stokes data to the characteristic polynomial of $\operatorname{mon}_\infty$ is bijective.

The local solution space $V(0)$ at $z=0$ has a similar description. The map
from the space of the Stokes matrices to the (nontrivial) coefficients of the characteristic polynomial of the topological monodromy $\operatorname{mon}_0$ at $z=0$, is bijective.

The monodromy space $\mathcal{R}_n$ consists of the local analytic data at $z=\infty$ and $z=0$ together with
a link which glues the solution space above $\mathbb{P}^1\setminus \{\infty \}$ to the solution space
above $\mathbb{P}^1\setminus \{0\}$. More precisely, the link $L\colon V(0)\rightarrow V(\infty)$ is a linear bijection such that
$L\circ (\operatorname{mon}_0)^{-1}=\operatorname{mon}_\infty \circ L$. The ``inverse sign'' reflects the difference in directions of the paths for
$\operatorname{mon}_0$ and $\operatorname{mon}_\infty$.

 It follows that all the structure of $V(0)$ is determined by $\operatorname{mon}_\infty$ and the link. In particular, $\operatorname{mon}_0^{-1}$ has the same characteristic polynomial as $\operatorname{mon}_\infty$. Furthermore, the Stokes matrices at~${z=0}$ are the same as those at $z=\infty$, however taken in the opposite order.

Let $L_0$ be a fixed choice for the link. Any other link has the form $M\circ L_0$ where
$M=(m_{i,j})\in {\rm GL}(V(\infty))$ commutes with $\operatorname{mon}_\infty$. Then $\mathcal{R}_n$ can be identified with the tuples $(M,x_1,\dots ,x_{n-1})$ as above and $M$ taken modulo multiplication by a scalar (since the basis of $V(0)$ and $V(\infty)$ can be
 scaled).

A computation, using the formulas in \cite[Sections~3.3 and~3.4]{CM-vdP} for $\operatorname{mon}_\infty$, shows that $M$ is determined by its last
row $(m_{n,1},\dots ,m_{n,n})$ and that this space
can be identified with the open subspace of $\mathbb{P}^{n-1}\times \mathbb{A}^{n-1}$ consisting
of the tuples
\[((m_{n,1}:\dots :m_{n,n}),(x_1,\dots ,x_{n-1}))\in \mathbb{P}^{n-1}\times \mathbb{A}^{n-1}\]
 such that
the determinant $F$ of the matrix $M$ is not zero. One easily sees that $F$ is homogeneous of degree
$n$ in the $n$ variables $m_{n,1},\dots ,m_{n,n}$ and its coefficients are polynomials in $x_1,\dots ,x_{n-1}$.
In particular, {\it $\mathcal{R}_n$ is smooth, connected, quasi projective of dimension $2(n-1)$}.

\begin{Example}[example~$\mathcal{R}_2$] The local analytic data at $z=\infty$ are
\[
V(\infty)=V(\infty)_{\sqrt{tz}}\oplus V(\infty)_{-\sqrt{tz}}=\mathbb{C}e_0+\mathbb{C}e_1,\qquad
\gamma \colon\ e_0\mapsto e_1 \mapsto -e_0.
\]
The singular directions depend on $t^{1/2}$. For $t^{1/2}$ in a neighbourhood
of $1$, the monodromy identity is
\[\operatorname{mon}_\infty= \left(\begin{matrix}0 & -1\\ 1 & 0\end{matrix}\right)\left(\begin{matrix}1&0\\ x & 1\end{matrix}\right) =
\left(\begin{matrix}-x &-1\\ 1 & 0\end{matrix}\right).
\]
\end{Example}

As above, there is a surjective morphism $\mathcal{R}_2\rightarrow \mathbb{A}^1={\rm Spec}(\mathbb{C}[x])$. The fibres consist of the $(b_3:b_4)\in \mathbb{P}^1$ such that the determinant
 $F=-b_3b_4x + b_3^2 + b_4^2$ of the matrix $M:=\bigl(\begin{smallmatrix}b_1 &b_2\\ b_3& b_4\end{smallmatrix}\bigr)$,
 commuting with $\operatorname{mon}_\infty$, is nonzero. Thus \smash{$\mathcal{R}_2\subset \mathbb{P}^1\times \mathbb{A}^1$} is the complement of the quadratic curve $F=0$ over $\mathbb{A}^1$.

 We note that the description in \cite[Section~3.6]{vdP-Sa} of monodromy space for the classical case~$P_{3}(D_8)$ is slightly different.
 There the link $L$ is normalized by the assumption $\det L=1$.

\begin{Example}[example~$\mathcal{R}_3$] The local analytic data at $z=\infty$ are
$V(\infty)=\mathbb{C}e_0+\mathbb{C}e_1+\mathbb{C}e_2$
with $\mathbb{C}e_j=V(\infty)_{q_j}$ for $j=0,1,2$ and \smash{$q_0=t^{1/3}z^{1/3}$}, \smash{$q_1=\zeta_3 t^{1/3}z^{1/3}$}, \smash{$ q_2=\zeta_3^2t^{1/3}z^{1/3}$} and
\smash{$\zeta_3 ={\rm e}^{2\pi {\rm i} /3}$}. The basis vectors $e_0$, $e_1$, $e_2$ are chosen such that
 the formal monodromy $\gamma$ satisfies $e_0\mapsto e_1\mapsto e_2\mapsto e_0$.
 The basis $e_0$, $e_1$, $e_2$ is unique up to a~simultaneous multiplication by a~scalar.
\end{Example}
For $t^{1/3}$ equal to $1$, the topological monodromy $\operatorname{mon}_\infty$ at $z=\infty$
is $\operatorname{mon}_{\infty}=\gamma {\rm St}_{3/4} {\rm St}_{1/4}$ which is explicitly
\[\operatorname{mon}_\infty=\begin{pmatrix}0 &0 &1 \\ 1&0 &0 \\ 0&1 &0 \end{pmatrix}
\begin{pmatrix}1 &0 &0 \\ 0& 1&0 \\ 0& x_{21}&1 \end{pmatrix}
\begin{pmatrix}1&x_{01} &0 \\ 0&1 &0 \\ 0& 0&1 \end{pmatrix}=
\begin{pmatrix} 0&x_{21} &1 \\ 1&x_{01} &0 \\ 0& 1&0 \end{pmatrix}.\]
 The characteristic polynomial of $\operatorname{mon}_\infty$ is $X^3-x_{01}X^2-x_{21}X-1$.

The space $\mathbb{A}^2$ of the topological monodromies at $z=\infty$, consists of the pairs $(x_{01},x_{21})\in \mathbb{C}^2$. The fibres of the obvious map $\mathcal{R}_3\rightarrow \mathbb{A}^2$ consist of the elements
$(m_{3,1}:m_{3,2}:m_{3,3})\in \mathbb{P}^2$ such that the determinant $F$ of $M$ is invertible. Explicitly,
 \begin{align*}
 F={}&a_7^3x_1x_2 + a_7^2a_8x_1^2 + a_7^2a_9x_2^2 + a_7a_8a_9x_1x_2
-a_7^2a_8x_2 + a_7^2a_9x_1 \\ &- 2a_7a_8^2x_1 + 2a_7a_9^2x_2 - a_8^2a_9x_2
 + a_8a_9^2x_1 + a_7^3 - 3a_7a_8a_9 + a_8^3 + a_9^3,
 \end{align*}
where $(a_7,a_8,a_9)=(m_{3,1},m_{3,2},m_{3,3})$ and $x_1=x_{0,1}$, $ x_2=x_{2,1}$.
Thus $\mathcal{R}_3\subset \mathbb{P}^2\times \mathbb{A}^2$ is the complement of the cubic curve
over $\mathbb{A}^2$ with equation $F=0$.

\subsection[Construction of M\_n by using a cyclic covering]{Construction of $\boldsymbol{\mathcal{M}_n}$ by using a cyclic covering}

The space $\mathcal{M}_n$ will be represented by the ``universal'' matrix differential operator
 $L=z\frac{{\rm d}}{{\rm d}z}+A$ of size $n\times n$ over $\mathbb{C}(z)$. This operator has only at $z=0$ and $z=\infty$ singularities
and these are given by the Galois orbits of $z^{-1/n}$ at $z=0$ and $tz^{1/n}$ at $z=\infty$.
As in Section~\ref{two}, we use the $n$-cyclic covering of $\mathbb{P}^1$ to produce explicit formulas.

Now $L$ is seen as a map on a vector space $V$ of dimension $n$ over $\mathbb{C}(z)$. Let $\sigma$ denote
the automorphism of $\mathbb{C}\bigl(z^{1/n}\bigr)$ over $\mathbb{C}(z)$, given by $\sigma \bigl(z^{1/n}\bigr)=\omega z^{1/n}$ with $\omega={\rm e}^{2\pi {\rm i} /n}$.
Then $\sigma$ acts as semi-linear map on $W:=\mathbb{C}\bigl(z^{1/n}\bigr)\otimes V$ and $L$ extends uniquely to
a derivation $D$ on $W$. The operator $D$ has no ramification.
Define the trace $\operatorname{tr}\colon W\rightarrow V$ by $tr(w)=\sum _{j=0}^{n-1}\sigma^j(w)$.

{\it We apply the method of} \cite[Chapter~12]{vdP-Sa}, {\it to construct a universal family} (see also
Section~\ref{section0}).
 One considers a basis $e_0,e_1,\dots ,e_{n-1}$ of $W$ over $\mathbb{C}\bigl(z^{1/n}\bigr)$ such that $\sigma$ acts by
 $e_0\mapsto e_1 \mapsto \cdots \mapsto e_{n-1}\mapsto e_0$ and $D$ has, with respect to this basis, poles of order 1
 at~\smash{$z^{1/n}=0$} and at \smash{$z^{1/n}=\infty$} and no further singularities. By construction, $D$ commutes with
$\sigma$. In particular, $D(e_0)$ determine~$D$ and $D(e_0)$ has the form
\smash{$\sum _{j=0}^{n-1}\bigl(a_jz^{-1/n}+b_j+c_jz^{1/n}\bigr)e_j$} where the $a_j$, $b_j$, $c_j$ are variables, parametrizing the family.

Define the basis $\{B_0,B_1,\dots ,B_{n-1}\}$ of $V$ by \smash{$B_j=\operatorname{tr}\bigl(z^{j/n}e_0\bigr)$} for all $j$.
In the computations, we change $B_{n-1}$
into $z^{-1}B_{n-1}$. The given data for $D(e_0)$ induces a formula $z\frac{{\rm d}}{{\rm d}z}+A$ for $D$ on the basis $B_0,\dots ,B_{n-1}$.

 It is seen that the matrix $A$ has at most singularities at $z=0$ and $z=\infty$ (in fact poles of order at most 1). The characteristic polynomial of $A$ is seen to have the form
 \[
 T^n+p_{n-1}T^{n-1}+\cdots +p_1T+p_0-\bigl(\alpha z^{-1}+\beta z\bigr)
 \]
 with all entries $p_0,\dots ,p_{n-1},\alpha , \beta$ are in $\mathbb{C}[a_0,b_0,c_0,\dots ,b_{n-1},c_{n-1}]$.
 In particular, there are explicit expressions $\neq 0$ for $\alpha$ and $\beta$. In the family given by $A$, we {\it require} that $\alpha$ and $\beta$
 are invertible. Indeed, this follows from the assumption that $z=0$ and $z=\infty$ are totally ramified and have Katz invariant
 $\frac{1}{n}$ for the operator $z\frac{{\rm d}}{{\rm d}z}+A$.

 The resulting affine family of operators $z\frac{{\rm d}}{{\rm d}z}+A$ is parametrized by
 \[
 \textrm{Spec}\left(\mathbb{C}\left[a_0,b_0,\dots ,c_{n-1},\frac{1}{\alpha}, \frac{1}{\beta}\right]\right).
 \]
 Next, we make the following
restrictions and normalizations. We {\it require} that $A$ has trace zero. This is equivalent to giving
$b_0$ the value \smash{$\frac{3-n}{2n}$}. The variable $z$ is scaled such that $\alpha=1$ and we write $t$ for $\beta$.
The next step is to divide by the action, by conjugation, of the group of the (constant) diagonal matrices on the differential operator $z\frac{{\rm d}}{{\rm d}z}+A$. In examples this is done by replacing $(n-1)$ suitable entries of $A$ by $1$ (for example, resulting in $a_1=1$, $a_2=\cdots =a_{n-1}=0$). One sees that the dimension of the final family (not counting $t$) is $3n-1-2-(n-1)=2n-2$ (with $-1$ for $b_0$ and $-2$ for $\alpha$, $\beta$ and $-(n-1)$ for conjugation).
This is in accordance with $\dim \mathcal{R}_n=2(n-1)$.

We do not attempt to describe the full moduli space $\mathcal{M}_n$, but claim that the constructed family
$D=z\frac{{\rm d}}{{\rm d}z}+A$ describes an affine open subset. The operator $E:=\frac{{\rm d}}{{\rm d}t}+B$ such that $\{D,E\}$ forms
a Lax pair is also considered as $\sigma$-equivariant operator on $W$ and is determined by $E(e_0)$. One computes that
$E(e_0)=z^{1/n}\sum _{j=0}^{n-1}c_je_j$ holds, under the assumption that $a_0=1$, $a_1=\dots =a_{n-1}=0$.
Below, the above construction is made explicit for $n=3$, extended to $n=4$ and to general $n\geq 3$. For $n=2$, it is compared
to the classical formula.

\subsubsection[Case n=3]{Case $\boldsymbol{n=3}$}
The matrix of $D$ with respect to the basis $B_0$, $B_1$, $z^{-1}B_2$ is
\[\begin{pmatrix} b_0+b_1+b_2 & a_0+a_1+a_2 & c_0+c_1+c_2 \\
 c_0+c_1\omega ^2 +c_2\omega &\frac{1}{3}+b_0+b_1\omega^2+b_2\omega
 & z^{-1}(a_0+a_1\omega^2+a_2\omega ) \\
a_0+a_1\omega+a_2\omega^2 & z(c_0+c_1\omega+c_2\omega ^2)& -1/3+b_0+b_1\omega +b_2\omega^2
\end{pmatrix} \]
with $\alpha =(a_0+a_1+a_2)\bigl(a_0+a_1\omega +a_2\omega ^2\bigr)\bigl(a_0+a_1\omega^2+a_2\omega\bigr)$ and
 $\beta=(c_0+c_1+c_2)\bigl(c_0+c_1\omega +c_2\omega ^2\bigr)\bigl(c_0+c_1\omega^2+c_2\omega\bigr)$.

 {\it Normalization}: $(a_0+a_1+a_2)=\bigl(a_0+a_1\omega +a_2\omega ^2\bigr)=\bigl(a_0+a_1\omega^2+a_2\omega\bigr)=1$ (equivalently $a_0=1$, $a_1=a_2=0$), $b_0=0$ and $\beta=t$.
This produces
\[z\frac{{\rm d}}{{\rm d}z}+\begin{pmatrix}d_0 &1& f_0\\ f_1&d_1 & \frac{1}{z}\\ 1 &f_2z &d_2\end{pmatrix} \qquad\text{with}\quad f_0f_1f_2=t\quad\text{and}\quad d_0+d_1+d_2=0.\]
 It is completed to a Lax pair by
\[
t\frac{{\rm d}}{{\rm d}t}+\begin{pmatrix}0 &0& f_0\\ f_1&0 &0 \\ 0 & f_2z&0 \end{pmatrix}.
\]
The Painlev\'e type equations are
\begin{alignat*}{4}
& t\frac{f'_0}{f_0}= d_0-d_2 ,\qquad&&  t\frac{f'_1}{f_1}=d_1-d_0 ,\qquad && t\frac{f'_2}{f_2}=d_2-d_1+1 ,& \\
& td_0'=f_1-f_0,\qquad && td_1'=f_2-f_1,\qquad && td_2'=f_0-f_2.&
\end{alignat*}
 The system of equations has symmetries $\rho$ and $\sigma$ defined by
$\rho\colon f_0,f_1,f_2 \mapsto f_1,f_2,f_0$ and $d_0,d_1,d_2 \mapsto d_1-1/3, d_2+2/3, d_0-1/3$,
$\sigma\colon f_0,f_1,f_2\mapsto f_1,f_0,f_2$ and $d_0,d_1,d_2 \mapsto -d_0,-d_2,-d_1$,
and generating $D_3=S_3$.

The above formulas are almost identical to the ones of Kawakami \cite[p.\ 35]{K3}. In the latter the
trace of the operator is $+1$ instead of $0$ and the entries of the matrix are written in terms of canonical variables
$p_1$, $p_2$, $q_1$, $q_2$ for a certain Hamiltonian.

One substitutes $F=f_0$, $G=f_1$ and obtains the equivalent system \[F''=\frac{(F')^2}{F}-\frac{F'}{t} +\frac{FG-2F^2}{t^2}+\frac{1}{tG},\]
 \[G''=\frac{(G')^2}{G}-\frac{G'}{t} +\frac{FG-2G^2}{t^2}+\frac{1}{tF}.\]

 For a solution $(f_0,f_1,f_2,d_0,d_1,d_2)$, invariant under the symmetry $f_0\leftrightarrow f_1$,
 one has $F=G$ and resulting equation
 \[
 F''=\frac{(F')^2}{F}-\frac{F'}{t} +\frac{-F^2}{t^2}+\frac{1}{tF}.
 \]

 The substitution $t=x^4, F(t)=xf(x)$ produces a solution $f$ of $P_3(D_8)$, i.e., $P_3$
 with parameters $(\alpha, \beta, \gamma, \delta)=(-16,0,0,16)$.
 The invariant solutions under $D_3=S_3$ are $F=G=\zeta t^{1/3}$ with~${\zeta^3=1}$.

\subsubsection{The general case} We make the following normalization:
\[
D(e_0)=\bigl(z^{-1/n}+b_0+c_0z^{1/n}\bigr)e_0+\sum _{j=1}^{n-1}\bigl(b_j+c_jz^{1/n}\bigr)e_{n-1},\qquad
b_0= \frac{3-n}{2n}, \qquad \beta =t
\] and
$E(e_0)=z^{1/n}\sum _{j=0}^{n-1}c_je_j$ and basis $B_0,\dots ,B_{n-2},z^{-1}B_{n-1}$.

For general $n\geq 3$, the formulas for the Lax pair are
\begin{gather*} z\frac{{\rm d}}{{\rm d}z}+ \begin{pmatrix} d_0 & 1 & 0 & . &0 & f_0\\
 f_1 &d_1 &1 &. &0 & 0 \\
 0 &f_2 &d_2 & . & 0 &0 \\
 .& . & . & . & 1 & .\\
 0 &. & . & f_{n-2} &d_{n-2} & \frac{1}{z}\\
 1 &0 & . & 0 &f_{n-1}z &d_{n-1} \end{pmatrix} ,\qquad
 t\frac{{\rm d}}{{\rm d}t}+ \begin{pmatrix} 0 & 0 & 0 & . &0 & f_0\\
 f_1 &0 &0 &. &0 & 0 \\
 0 &f_2 &0 & . & 0 &0 \\
 .& . & . & . & 0 & .\\
 0 &. & . & f_{n-2} &0 & 0\\
 0 &0 & . & 0 &f_{n-1}z &0 \end{pmatrix} \end{gather*}
 with $\sum d_j=0$, $ \prod f_j=t$. {\it The Painlev\'e type equations are}
\begin{alignat*}{5}
& t\frac{f'_0}{f_0}=d_0-d_{n-1},\qquad && t\frac{f_1'}{f_1}=d_1-d_0,\qquad&&\dots ,\qquad&& t\frac{f_{n-1}'}{f_{n-1}}=d_{n-1}-d_{n-2} +1,&\\
&td_0'=f_1-f_0,\qquad &&td_1'=f_2-f_1,\qquad &&\dots ,\qquad &&td_{n-1}'=f_0-f_{n-1}.&
\end{alignat*}

 The symmetries observed for $n=3$ generalize to $n\geq 3$ as follows
\begin{align*}
\rho\colon\ (f_0, f_1,\dots, f_{n-1})
& \mapsto (f_1,f_2,\dots , f_{n-1},f_0),\\
(d_0,d_1,\dots, d_{n-1})& \mapsto \left(d_1-\frac1n, \dots,d_{n-2}-\frac1n,d_{n-1}+\frac{n-1}{n}, d_0-\frac1n\right)
\end{align*}
and
\begin{gather*}
\sigma\colon\ (f_0, f_1,\dots, f_{n-1})
\mapsto (f_{n-1}, f_{n-2},\dots, f_1, f_0),
\end{gather*}
combined with $\sigma(d_j)=-d_{\pi(j)}+c_j$ for a permutation $\pi$ satisfying $\pi^2=1$ and constants $c_j\in\bigl\{-\frac1n,\frac{n-1}{n}\bigr\}$
such that $\sum c_j=0$. These symmetries generate the dihedral group $D_n$ of order $2n$.

Taking $D_n$-invariants $f_0=\dots=f_{n-1}:=t^{1/n}$ and corresponding $d_j$'s produces algebraic solutions of the Painlev\'{e} type equations.

 {\it Case $n=4$}. For invariant solutions under
 $f_0\leftrightarrow f_1$, $f_2\leftrightarrow f_3$
 one has $(f_0,f_1,f_2,f_4)=\smash{\bigl(f,f,\frac{\sqrt{t}}{f}, \frac{\sqrt{t}}{f}\bigr)}$ and
 $d_0=-1/8$, $d_1=t\frac{f'}{f}-1/8$, $d_2=3/8$, $d_3=-t\frac{f'}{f}-1/8$ and the equation
 \[ f''=\frac{(f')^2}{f}-\frac{f'}{t}-\frac{f^2}{t^2}+\frac{1}{t^{3/2}}.\]
 Substitution $t=x^4$, $f(t)=xF(x)$ yields the equation $P_3(D_8)$ for $F$.

 The $D_4$-symmetric solutions are
 $f=f_0=f_1=f_2=f_3=\zeta t^{1/4}$ with $\zeta^4=1$.

 {\it Case $n=5$}. Consider solutions invariant under
 $f_0 \leftrightarrow f_4$, $f_1\leftrightarrow f_3$. Then
 $(f_0,f_1,f_2,f_3,f_4)=\bigl(f,g,\frac{t}{f^2g^2},g,f\bigr)$. Moreover,
 \[
 (d_0,d_1,d_2,d_3,d_4)=
\left(t\frac{f'}{f}-\frac25, t\frac{f'}{f}+t\frac{g'}{g}-\frac25, -t\frac{f'}{f}-t\frac{g'}{g}+\frac35, t\frac{f'}{f}+\frac35, \frac{-2}{5} \right).\]
The system of equations for $f$, $g$ reads
\[f''=\frac{(f')^2}{f}-\frac{f'}{t}+\frac{fg-f^2}{t^2},\qquad g''=\frac{(g')^2}{g}-\frac{g'}{t}+\frac{fg-2g^2}{t^2}+\frac{1}{f^2gt}. \]
The $D_5$-invariant solutions are $f=g=\zeta t^{1/5}$ with $\zeta^5=1$.

The presented examples for small $n$ suggest that
subgroups of $D_n$ produce interesting subsystems.

 {\it For completeness, we consider also the case $n=2$}.
The Lax pair is
\[
z\frac{{\rm d}}{{\rm d}z}+\begin{pmatrix}
d_0&\frac{1}{z}+f_0\\
1+f_1z&d_1\end{pmatrix},\qquad
t\frac{{\rm d}}{{\rm d}t}+\begin{pmatrix}0& f_0\\
f_1z & 0
\end{pmatrix}
\]
 with
$f_0f_1=t$ and $d_0+d_1=0$. The equations are
\[t\frac{f_0'}{f_0}=d_1-d_0,\qquad t\frac{f_1'}{f_1}=d_0-d_1+1, \qquad td_0'=f_1-f_0,\qquad td_1'=f_0-f_1.\]
One observes that $q=f_0$ satisfies the classical equation for $P_3(D_8)$, namely{\samepage
\[
q''=\frac{(q')^2}{q}-\frac{q'}{t}+\frac{2q^2}{t^2}-\frac{2}{t}.
\]
There is a symmetry $\rho$, given by $f_0,f_1\mapsto f_1,f_0$ and $d_0, d_1\mapsto d_1-1/2, d_0+1/2$.}

 The symmetry means that if $q$ is a solution, then so is $\frac{t}{q}$. Further the invariant element \smash{$r=q+\frac{t}{q}$} satisfies the equation
 \[ r''=\frac{r}{r^2-4t}(r')^2-\frac{r^2}{t\bigl(r^2-4t\bigr)}r'+\frac{2r^4-16tr^2+tr+32t^2}{t^2\bigl(r^2-4t\bigr)}. \]
 The symmetric solutions are $q=\pm \sqrt{t}$ and $r=\pm 2\sqrt{t}$.

\section[A companion of P\_1]{A companion of $\boldsymbol{P_1}$} \label{thirteen}

What we like to call the {\it companion of $P_1$} is the family $\mathcal{M}$ of connections, given by the set of differential modules $M$ over
$\mathbb{C}(z)$ with dimension 2, $\Lambda ^2M$ is trivial, $z=0$ is regular singular and the generalized eigenvalues at $z=\infty$ are $\pm w$ with
$w=z^{5/2}+\frac{t}{2}z^{1/2}$. This is the $P_1$ case except for allowing a regular singularity at $z=0$.

 {\it Description of $\mathcal{R}$}. The singular directions at $z=\infty$, lying in $[0,1)$
are $\frac{1}{5}$, $\frac{3}{5}$ for the difference of eigenvalues $w-(-w)$, and $0$, $\frac{2}{5}$, $\frac{4}{5}$ for
$(-w)-w$. Thus $\mathcal{R}\cong \mathbb{A}^5$. Let $\operatorname{mon}\colon \mathcal{R}\rightarrow {\rm SL}_2(\mathbb{C})$
 denote the morphism which sends the Stokes matrices to the monodromy matrix at $z=0$.

 The fibre of $\operatorname{mon}$ above $\bigl(\begin{smallmatrix}a& b\\c& d\end{smallmatrix}\bigr)\in {\rm SL}_2$ is given by the
 monodromy identity
 \[ \begin{pmatrix}0& -1\\1&0 \end{pmatrix}\begin{pmatrix} 1& 0\\x_5& 1 \end{pmatrix}\begin{pmatrix}1& x_4\\0& 1\end{pmatrix}\begin{pmatrix}1& 0\\x_3& 1 \end{pmatrix}
\begin{pmatrix}1& x_2\\0& 1 \end{pmatrix}\begin{pmatrix}1& 0\\x_1& 1\end{pmatrix}= \begin{pmatrix}a& b\\c& d \end{pmatrix}. \]
One eliminates $x_1$, $x_2$ by $x_1=-c(x_3+x_5+x_3x_4x_5)-a(1+x_3x_4)$ and $x_2=d(1+x_4x_5)+bx_4$.
Since $ad-bc=1$ we are left with the equation $ d(x_3+x_5+x_3x_4x_5)+b(1+x_3x_4)+1=0$.

For $d\neq 0$, the fibre is (as often) an affine cubic surface with three lines at infinity. Its equation coincides with the one for the
standard family which produces the classical $P_1$ (see \cite[Section~3.10]{vdP-Sa}).

For $d=0$, the equation of the fibre reads $x_3x_4=-b^{-1}-1$ and $x_5$ has no relations.
For $b\neq -1$ it is the surface $\mathbb{C}^*\times \mathbb{C}$. In particular, $\mathcal{R}\rightarrow {\rm SL}_2$ is surjective. Furthermore, the parameter space~$\mathcal{P}$ has dimension 1.

 {\it Description of $\mathcal{M}$}. Since the fibres of ${\rm RH}\colon\mathcal{M}\rightarrow \mathcal{R}$ are
parametrized by $t$, one has $\dim \mathcal{M}=6$. The isomorphic classes of the residue matrix at $z=0$ form the parameter space.
It can be shown that the family
 \[ z\frac{{\rm d}}{{\rm d}z}+ \begin{pmatrix}0 & 1\\0& 0\end{pmatrix} z^3+
\begin{pmatrix} 0 &b_2\\1&0\end{pmatrix}z^2+
\begin{pmatrix} a_1 &b_1\\ c_1& -a_1\end{pmatrix} z+ \begin{pmatrix}a_0 & b_0 \\c_0& -a_0\end{pmatrix},\]
with $t=b_1-b_2^2+c_0$, is the universal family of connections $\mathcal{M}$. We eliminate $b_1$ by $b_1=t+b_2^2-c_0$.
Furthermore, $p_0:=a_0^2+b_0c_0$ is the basic parameter (i.e., independent of $t$ in an isomonodromic family).

For the Lax pair computation, we suppose that the above operator commutes with
\[
\frac{{\rm d}}{{\rm d}t}+\begin{pmatrix} y_1&y_2\\y_3& -y_1 \end{pmatrix}+z\begin{pmatrix} y_4& y_5\\y_6& -y_4 \end{pmatrix}
\]
 and that \smash{$\frac{d(a_0^2+b_0c_0)}{dt}=0$}.
This eliminates $y_1,\dots ,y_6$ and produces the equations
$a_0'=2b_2c_0-b_0$, $ b_0'=-4a_0b_2$, $ a_1'=-3b_2^2+2c_0-t$, $ b_2'=-2a_1$, $ c_0'=2a_0$.

 For $c_0=0$, and also for a fixed residue matrix (i.e., $a_0'=b_0'=c_0'=0$), one obtains the $P_1$ equation.
For $c_0\neq 0$, one eliminates $b_0=\bigl(p_0-a_0^2\bigr)/c_0$ and
$b_0'=\bigl(-2a_0a_0'c_0-\bigl(p_0-a_0^2\bigr)c_0'\bigr)/c_0^2$. This results in the Painlev\'e type vector field
\[a_0'=2b_2c_0-\frac{p_0-a_0^2}{c_0},\qquad c_0'=2a_0,\qquad a_1'=-3b_2^2+2c_0-t,\qquad b_2'=-2a_1.\]
One eliminates $a_0$, $a_1$, and $c_0$ in the above equations by
$a_1=-\frac{1}{2}b_2'$, $c_0=\frac{t}{2}+\frac{3}{2}b_2^2-\frac{1}{4}b_2''$, \smash{$a_0=\frac{1}{4}+\frac{3}{2}b_2b_2'-\frac{1}{8}b_2^{(3)}$}. The remaining equation produces the following fourth-order explicit differential equation for $f:=b_2$
 \begin{gather*}
 -2\bigl(6f^2-f^{(2)}+2t\bigr)f^{(4)}=288f^5-240f^3f^{(2)}+192tf^3-24ff^{(1)}f^{(3)}+32f\bigl(f^{(2)}\bigr)^2\\
 \qquad-80tff^{(2)}+32ft^2+
 24\bigl(f^{(1)}\bigr)^2f^{(2)} -48 \bigl(f^{(1)}\bigr)^2+48ff^{(1)}+\bigl(f^{(3)}\bigr)^2-4f^{(3)}+64p_0+4
\end{gather*}
with $f=b_2$, $f^{ (j) }:=\left( \frac{{\rm d}}{{\rm d}t} \right)^j (b_2)$ for $j=1,2,3,4$.
We note that the denominator of the formula for $f^{(4)}$ is the equation for $P_1$. It seems probable, but we have no proof, that
the field $\mathbb{C}(t)\bigl(b_2,\frac{{\rm d}}{{\rm d}t}b_2,\bigl(\frac{{\rm d}}{{\rm d}t}\bigr)^2b_2,\bigl(\frac{{\rm d}}{{\rm d}t}\bigr)^3b_2\bigr)$ has, for generic $p_0$, transcendence degree 4 over
$\mathbb{C}(t)$. Indeed, this would fit with the observation that the fibres of
$\mathcal{R}\rightarrow \mathcal{P}$ have dimension 4.

 {\it Comments}. There are two reasons why this ``companion of $P_1$'' is not in the classical list
$P_1-P_6$. The Painlev\'e type equations describe in fact a vector field of rank 4
(written above as explicit differential equation of order 4).

Secondly, the monodromic family is a subfamily of the natural monodromic family with ``two time variables''
given by the data: $z=0$ is regular singular and $z=\infty$ is irregular singular with generalized eigenvalues
$\pm \bigl(z^{5/2}+\frac{t_1}{2}z^{3/2}+\frac{t_2}{2}z^{1/2}\bigr)$ with time variables $t_1$, $t_2$. We extend our computations to this case.

{\it Isomonodromy and Lax pairs for $q=\bigl(z^{5/2}+\frac{t_1}{2}z^{3/2}+\frac{t_2}{2}z^{1/2}\bigr)$.}
As in the case \smash{$q=z^{5/2}+\frac{t}{2}z^{1/2}$} the family of connections $z\frac{{\rm d}}{{\rm d}z}+A$ can be normalized to
\[ z\frac{{\rm d}}{{\rm d}z}+ \begin{pmatrix} 0 & 1 \\ 0& 0\end{pmatrix} z^3+
\begin{pmatrix} 0 & b_2 \\ 1& 0\end{pmatrix} z^2+
\begin{pmatrix} a_1 &b_1 \\ c_1& -a_1\end{pmatrix} z+ \begin{pmatrix} a_0 & b_0
 \\ c_0& -a_0\end{pmatrix},\]
where $c_0=b_2^2-b_2t_1+\frac{1}{4}t_1^2-b_1+t_2$ and $c_1=-b_2+t_1$.
The variables $a_0$, $a_1$, $b_0$, $b_1$, $b_2$ are seen as functions of $t_1$, $t_2$. The Lax pairs are expressed
by $\bigl[z\frac{{\rm d}}{{\rm d}z}+A, \frac{{\rm d}}{{\rm d}t_i}+B_i\bigr]=0$ for $i=1,2$ and~$B_i$ a~matrix depending on $t_1$, $t_2$, $z$ and polynomial in $z$
of degree $\leq 2$. One obtains in terms of closed one-forms $d(a_0),\dots ,d(b_2)$ the system
\begin{align*}
{\rm d}(a_0)={} &\bigl\{16b_2^4-16b_2^3t_1+4b_2t_1^3-t_1^4-48b_1b_2^2+32b_1b_2t_1-4b_1t_1^2\\
&+32b_2^2t_2-16b_2t_1t_2-16b_0b_2+
 8b_0t_1+32b_1^2-48b_1t_2+16t_2^2 \bigr\} \frac{{\rm d}t_1}{48} \\
 & +\left\{3b_2^2t_1 -2b_2^3 -\frac{3b_2t_1^2}{2}+\frac{t_1^3}{4}+2b_1b_2 -b_1t_1 -2b_2t_2+t_1t_2+b_0 \right\} {\rm d}t_2, \\
{\rm d}(a_1)={} &\bigl\{
 -16b_2^3+20b_2^2t_1-4b_2t_1^2-t_1^3+16b_1b_2-16b_1t_1-16b_2t_2\\ &+12t_1t_2+8b_0\bigr\}\frac{{\rm d}t_1}{24} +
 \left\{b_2^2-2b_2t_1+\frac{3}{4}\bigl(t_1^2\bigr)+2b_1-t_2 \right\} {\rm d}t_2, \\
{\rm d}(b_0)={} &\bigl\{a_0t_1^2-4a_0b_2^2+8a_0b_1-4a_0t_2-4a_1b_0\bigr\}\frac{{\rm d}t_1}{6}+\{ (4b_2-2t_1)a_0 \}{\rm d}t_2, \\
{\rm d}(b_1)={} &\bigl\{ -4a_1b_2^2+a_1t_1^2+4a_0b_2-2a_0t_1+4a_1b_1-4a_1t_2+2b_2-t_1\bigr\} \frac{{\rm d}t_1}{6}\\
 &+\{ 4a_1b_2-2a_1t_1+2a_0+1 \} {\rm d}t_2,\\
 {\rm d}(b_2)={} & \{ -a_1t_1+2a_0+2\}\frac{{\rm d}t_1}{3}+2a_1{\rm d}t_2.\end{align*}
Note that $p_0:=a_0^2+b_0c_0$ satisfies ${\rm d}(p_0)=0$ and $p_0$ is a generating parameter for this system.

\subsection*{Acknowledgements} We thank the referees for their work and suggestions, resulting in a considerable revision of
an earlier version of this text. We especially thank Anton Dzhamay for his successful effort to identify two of the Painlev\'e type
equations obtained in this paper.

\pdfbookmark[1]{References}{ref}
\LastPageEnding

\end{document}